\documentclass[a4paper, reqno]{amsart}
\usepackage{amsthm}
\usepackage{amsmath}
\usepackage{amssymb}
\usepackage{hyperref}
\usepackage{tikz-cd}

\title[A model for equivariant commutative ring spectra]{A model for genuine equivariant commutative ring spectra away from the group order}
\author{Christian Wimmer}
\date{\today} 
\address{Universit\"at Hamburg, Fachbereich Mathematik, Bundesstra\ss{}e 55, 20146 Hamburg, Germany}
\email{christian.wimmer@uni-hamburg.de}
\keywords{Commutative ring spectra, equivariant homotopy theory, geometric fixed points}
\subjclass[2010]{55P91}

\newcommand{\bn}{\mathbb N}
\newcommand{\bp}{\mathbb P}
\newcommand{\bbq}{\mathbb Q}
\newcommand{\br}{\mathbb R}
\newcommand{\bs}{\mathbb S}
\newcommand{\bu}{\mathbb U}
\newcommand{\bz}{\mathbb Z}

\newcommand{\cc}{\mathcal C}
\newcommand{\cd}{\mathcal D}
\newcommand{\cf}{\mathcal F}
\newcommand{\cm}{\mathcal M}
\newcommand{\cn}{\mathcal N}

\newcommand{\cu}{\mathcal U}

\newcommand{\MF}{\mathcal{MF}}

\newcommand{\lra}{\longrightarrow}
\newcommand{\ra}{\rightarrow}

\newcommand{\tensor}{\otimes}

\newcommand{\regrep}[1][G]{\rho_{#1}}

\newcommand{\orthspec}{\Sp^O}

\newcommand{\gensp}[1][G]{#1\hyph\Sp}

\newcommand{\suspsp}{\Sigma^{\infty}}

\newcommand{\orbcat}[1][G,\fin]{\Orb_{#1}}
\newcommand{\orbcatinv}[1][G,\fin]{\Orb^\times_{#1}}

\newcommand{\preorbcat}[1][G,\fin]{\widetilde{\Orb}_{#1}}
\newcommand{\preorbcatinv}[1][G,\fin]{\widetilde\Orb^\times_{#1}}

\DeclareMathOperator{\Aut}{Aut}
\DeclareMathOperator{\bigsmashp}{\bigwedge}

\DeclareMathOperator{\CDGA}{CDGA}
\DeclareMathOperator{\Ch}{Ch}
\DeclareMathOperator{\Com}{Com}
\DeclareMathOperator{\colim}{colim}

\DeclareMathOperator{\cts}{cts}

\DeclareMathOperator{\ev}{ev}

\DeclareMathOperator{\Fun}{Fun}
\DeclareMathOperator{\fin}{fin}
\DeclareMathOperator{\Fin}{Fin}
\DeclareMathOperator{\gen}{gen}
\DeclareMathOperator{\Ho}{Ho}
\DeclareMathOperator{\Hom}{Hom}
\DeclareMathOperator{\hocolim}{hocolim}
\DeclareMathOperator{\hyph}{-}
\DeclareMathOperator{\Id}{Id}
\DeclareMathOperator{\kernel}{ker}
\DeclareMathOperator{\Lan}{Lan}
\DeclareMathOperator{\Lin}{L}

\DeclareMathOperator{\Map}{Map}
\DeclareMathOperator{\map}{map}
\DeclareMathOperator{\N}{N}
\DeclareMathOperator{\Nat}{Nat}
\DeclareMathOperator{\Ncoh}{N_{\Delta}}

\DeclareMathOperator{\Orb}{Orb}
\DeclareMathOperator{\orth}{O}

\DeclareMathOperator{\op}{op}
\DeclareMathOperator{\pr}{pr}
\DeclareMathOperator{\res}{res}

\DeclareMathOperator{\sh}{sh}
\DeclareMathOperator{\SH}{SH}
\DeclareMathOperator{\Stab}{Stab}
\DeclareMathOperator{\Sp}{Sp}
\DeclareMathOperator{\smashp}{\wedge}
\DeclareMathOperator{\Sym}{Sym}
\DeclareMathOperator{\Top}{Top}
\DeclareMathOperator{\tr}{tr}

\numberwithin{equation}{section}
\newtheorem{Corollary}[equation]{Corollary}
\newtheorem{Lemma}[equation]{Lemma}
\newtheorem{Proposition}[equation]{Proposition}
\newtheorem{Theorem}[equation]{Theorem}

\theoremstyle{definition}
\newtheorem{Definition}[equation]{Definition}

\theoremstyle{remark}
\newtheorem{Remark}[equation]{Remark}

\begin{document}

\maketitle

\begin{abstract}
We use geometric fixed points to describe the homotopy theory of genuine equivariant commutative ring spectra after inverting the group order.
The main innovation is the use of the extra structure provided by the Hill-Hopkins-Ravenel norms in the form of additional norm maps on geometric fixed point diagrams, which turns out to be computationally managable.
\end{abstract}

\tableofcontents


\section{Introduction}\label{intro}
While the stable homotopy category $\SH$ is of vast computational complexity,
dramatic simplifications occur upon rationalizing it (i.e.\ restricting to spectra whose homotopy groups are rational vector spaces)
and the resulting category is equivalent to graded $\mathbb Q$-vector spaces (which can be thought of as a classification of rational cohomology theories).
It is structurally much simpler than what we started with, but still retains interesting information.
This is only amplified when considering (the various versions of) $G$-equivariant stable homotopy theory for a finite group $G$.
The complexity of course increases with that of the group, but one can still give rational algebraic models.
Moreover, this can be done in a sufficiently multiplicative fashion yielding descriptions of rational $A_\infty$- and $E_\infty$-ring spectra in terms of (commutative) differential graded algebras.
In the following we consider \emph{genuine} equivariant stable homotopy theory,
the version with the richest structure.
In terms of cohomology theories, this corresponds to $\operatorname{RO}(G)$-graded theories on $G$-spaces. 
There are many 'levels of commutativity' a ring $G$-spectrum can enjoy and a genuine-commutative ring spectrum is the most sophisticated version
in that sense. Interest in genuine equivariant (ring-) spectra has been especially reinvigorated with their use in the solution of the Kervaire invariant one problem by Hill-Hopkins-Ravenel \cite{kervaire}.
What has been missing so far is an algebraic model for rational genuine commutative ring spectra and in this article we largely finish the equivariant story (for finite groups) by giving one.
In fact, we go somewhat further by working relative to a family of finite subgroups of a fixed compact Lie group $G$.

To put this into context, we briefly outline the classical situation in more detail.
Serre's computation of rational homotopy groups shows that taking stable homotopy groups induces an equivalence
\[
	\SH_{\bbq}\overset{\simeq}{\longrightarrow}\operatorname{grVect}_\bbq
\]
between the rational stable homotopy category and graded $\bbq$-vector spaces, 
which is symmetric monoidal with respect to the smash product and the graded tensor product.
In particular, one also obtains an algebraic model for rational homotopy-commutative ring spectra, namely graded-commutative $\bbq$-algebras.

To deal with more highly structured ring spectra, the comparison has to take place on the level of homotopy theories
(e.g.\ model categories or underlying $\infty$-categories), even if one just wants to describe the homotopy category.
The above functor can be refined to a \emph{chain functor}
\(
	\Sp \rightarrow \Ch_\bbq
\)
taking value in rational chain complexes (coming with a natural equivalence $H_\ast(CX)\cong \pi_\ast(X)\otimes \bbq$).
This is straightforward to construct by either writing it down explicitely or invoking the universal property of the $\infty$-category of spectra.
With some more work one also obtains a (homotopically) symmetric monoidal functor (unpublished work by Schwede-Strickland or $\infty$-categorically using \cite{NikolausOperads}).

This yields an equivalence
\[
	\Com(\Sp)_\bbq \simeq \CDGA_\bbq
\]
between the rational homotopy theories of commutative ring spectra and commutative differential graded algebras
(see \cite{CommutativeHZ-Algebras} for a zig-zag of Quillen equivalences). A more invariant statement
(which now automatically follows from general principles of symmetric monoidal $\infty$-categories)
is that there is a rational equivalence between $E_\infty$-ring spectra and $E_\infty$-algebras.  These are modelled by strict commutative algebra objects on the pointset level (the former even integrally).
We note that this is no longer true in the equivariant setting,
where genuine commutative rings (still modelled by strict commutative algebra objects) encode more information than $E_\infty$-rings.

Additively, the equivariant version (for a finite group $G$) of this story is similar,
although the bookkeeping becomes more involved. As the subgroup $H\leq G$ varies,
the collection of homotopy groups $\underline\pi_\ast X=\{\pi_\ast^H(X)\}$ of a genuine $G$-spectrum form 
a \emph{Mackey functor}, that is the values are related by restriction, transfer, and conjugation maps satisfying certain relations.
Like in the non-equivariant case, taking homotopy groups yields a rational equivalence
\[
\SH(G)_{\bbq}\overset{\simeq}{\longrightarrow}\operatorname{grMF(G)}_\bbq
\]
between the genuine equivariant stable homotopy category and graded Mackey functors
\cite[Appendix A]{GeneralizedTate}, but lifting this to a symmetric monoidal chain functor is a more subtle task.
Luckily, rational Mackey functors can be replaced up to equivalence by a much simpler abelian category.
The category splits as a product
\[
	\MF(G)_\bbq\simeq \prod_{(H\leq G)}\bbq[W_GH]\hyph\operatorname{mod}
\]
of module categories over the Weyl groups $W_GH$, indexed by the conjugacy classes of subgroups $H$
(in fact, this is already true after inverting the group order).
It is a folklore result that this splitting admits a spectral lift
\[
	G\hyph\Sp\overset{\Phi}{\longrightarrow}\prod_{(H\leq G)}\bs[W_GH]\hyph\operatorname{mod}
\]
given by \emph{geometric fixed points}, a nicely behaved version of taking fixed points in equivariant stable homotopy theory.
The $G$-geometric fixed points $\Phi^G:G\hyph\Sp\rightarrow \Sp$ are
characterised up to equivalence by the following properties: They are lax symmetric monoidal,
commute with homotopy colimits and have the 'correct' value $\Phi^G(\Sigma^\infty A)\simeq \Sigma^\infty A^G$ on the suspension spectra of $G$-spaces. As the subgroup $H$ varies, the collection of $H$-geometric fixed points can be assembled into a symmetric monoidal functor as above.
However, as we explain below this does not yield a description of genuine-commutative ring spectra by simply passing to commutative algebra objects.

\subsubsection*{Summary of results}
Let $G$ be a compact Lie group, $\cf$ a family of finite subgroups of $G$, and $R\subseteq \bbq$ a ring such that
the groups orders in $\cf$ are invertible. Let $\Orb_\cf$ be the full (topological) subcategory of the orbit category spanned by the transitive $G$-sets $G/H$ for $H\in \cf$
and $\Orb_\cf^\times$ the subcategory with only isomorphisms. We note that there is an equivalence
$\Orb_\cf^\times\simeq \coprod_{(H\in F)} B(W_GH)$, so that the product of module categories above can be rewritten as a diagram category
\[
	\Orb_\cf^\times\hyph\Sp\simeq\prod_{(H\in \cf)}\bs[W_GH]\hyph\operatorname{mod}
\]
indexed by the invertible orbit category. We prefer the 'coordinate-free' description because it allows for a more conceptual way to express the additional data needed to encode commutative ring spectra. The first step in obtaining a more concrete $R$-local model for these is the following additive result.
Although not surprising to the experts, it does not seem to have been recorded in the literature before in this level of generality
(working relative to an ambient compact Lie group and only inverting the groups orders in the family) and making specific use of geometric fixed points.

\begin{Theorem}
Geometric fixed points induce a symmetric monoidal equivalence
	\[
	(G\hyph\Sp)_{\cf,R}\simeq \Orb^\times_\cf\hyph \Sp_R
	\]
between the $R$-local homotopy theories of genuine equivariant $\cf$-spectra and $\Orb_\cf$-diagrams in spectra.
\end{Theorem}
As a corollary, by passing to commutative algebra objects on the level of $\infty$-categories
we obtain an equivalence
\[
	E_{\infty}(G\hyph\Sp)_{\cf,R}\simeq \Orb^\times_\cf\hyph E_{\infty}(\Sp_R)
\]
between naive-commutative equivariant ring spectra
(modelled by algebras over an $E_\infty$-operad equipped with trivial $G$-action) and $\Orb_G^\times$-diagrams in non-equivariant
$E_{\infty}$-rings (modelled by commutative ring spectra on the pointset level), see \cite{NaiveCommutative} for a zig-zag of Quillen equivalences in the case where $G$ is finite.

However, the functor $\Phi:(G\hyph\Sp)_{\cf,R}\rightarrow \Orb^\times_\cf\hyph \Sp_R$
has no chance of 'restricting' to an identification of  genuine commutative ring spectra without equipping it with additional structure.
It might seem counterintuitive at first sight, but only when viewed purely on the point-set level.
Equivariantly, the symmetric monoidal structure on the model for genuine $G$-spectra (i.e.\ equivariant orthogonal spectra) encodes more homotopical information than the underlying $\infty$-categorical symmetric monoidal structure, 
namely the Hill-Hopkins-Ravenel norms.
In terms of geometric fixed points of commutative ring spectra, this can be expressed in the form of additional \emph{norm maps} 
\[
	N_H^K:\Phi^HR\rightarrow \Phi^KR
\]
for inclusions $H\leq K$ of finite subgroups of $G$. Working with a suitable point-set level model for geometric fixed point spectra,
we produce a functor
\[
	\Phi:\Com(G\hyph\Sp)_{\cf,R}\rightarrow \Orb_\cf\hyph \Com(\Sp)_R
\]
with values in diagrams indexed by the full orbit category, where the additional projections $G/H\rightarrow G/K$ encode the norm maps.
\begin{Theorem}\label{thm:ModelCommutative}
Geometric fixed points induce an equivalence
\[
\Com(G\hyph\Sp)_{\cf,R}\simeq \Orb_\cf\hyph \Com(\Sp)_R
\]
between the $R$-local homotopy theories of genuine equivariant $\cf$-complete commutative ring spectra and $\Orb_\cf$-diagrams in commutative ring spectra.
\end{Theorem}
\begin{Remark}
The algebraic precursor to this result is due to Strickland, who shows that there is an equivalence between the categories of rational Tambara functors and $\Orb_G$-diagrams in commutative $\mathbb Q$-algebras (\cite[17.2]{StricklandTambara}).
This can now be viewed as the $\pi_0$-shadow of the topological statement.
\end{Remark}

In particular, for a finite group $G$ we obtain an algebraic rational model:
\begin{Corollary}
Geometric fixed points induce a rational equivalence 
\[
\Com(G\hyph\Sp)_\bbq\simeq \Orb_G\hyph \CDGA_{\bbq}
\]
with $\Orb_G$-diagrams in commutative differential graded algebras.
\end{Corollary}

We briefly sketch the strategy employed to demonstrate Theorem 1.2.
Once the geometric fixed point functors are constructed, the relevant data can be displayed in the following diagram 
\begin{center}
	\begin{tikzcd}[row sep = large, column sep = large]
	(\gensp)_{\mathcal F, R}	\arrow[r, "\Phi"]\arrow[d, shift right = 2, "\mathbb P", swap]	&
	\Orb^{\times}_{\mathcal F}\hyph \Sp_R	\arrow[d, shift right = 2, swap, "\mathbb P\circ \iota_!"]		\\
	
	\Com(\gensp)_{\mathcal F, R}	\arrow[r, "\Phi"]\arrow[u, shift right = 2, "\mathbb U", swap]	&
	\Orb_{\mathcal F}\hyph \Com(\Sp)_R	\arrow[u, shift right = 2, "\mathbb U", swap],
	\end{tikzcd}
\end{center}
where the vertical pairs are free-forgetful adjunctions. These are monadic (in the higher categorical sense), which is somewhat tedious to show from a technical point of view.
The square with the forgetful functors commutes (up to equivalence) essentially by construction and the upper horizontal functor is an equivalence (Theorem 1.1).
We wish to conclude that the lower horizontal arrow is also an equivalence.
While all categories and functors are already defined at the pointset level (i.e.\ the $\infty$-categories are presented by model categories and the functors are induced by homotopical functors or Quillen functors), it is at this point that things work out much cleaner in the higher categorical setting, because we can apply the 'monadic comparison theorem'.
This says that in this situation it will suffice to show that geometric fixed points preserve free commutative ring spectra, i.e.\ the square with the free functors commutes.
This is the most non-formal input and amounts to a computation of geometric fixed point of free ring ring spectra.
\subsubsection*{Generalizations}
Naive-commutative and genuine-commutative $G$-ring spectra lie at opposite ends of a whole family of highly structured equivariant ring spectra.
The latter can also be described as algebras over $G\hyph E_\infty$-operads, i.e.\ $G$-operads $\mathcal O$ such that the $n$-ary operations
$\mathcal O(n)\simeq E_G\Sigma_n$ form a universal space for the family of \emph{graph subgroups} $\Gamma\leq G\times \Sigma_n$
(the subgroups arising as the graph of a morphism $G\geq H\overset{\alpha}{\rightarrow}\Sigma_n)$. Let $G$ be a finite group. 
In \cite{noperads} Blumberg and Hill introduce a more general class of $N_\infty$-operads,
which interpolate between $E_\infty$- and $G\hyph E_\infty$-operads. 
Algebras over an $N_\infty$-operad admit norm maps along certain subgroup inclusions, which are called \emph{admissible}.
In the above extreme cases there are no norm maps respectively norms along all subgroup inclusions.
In light of Theorem \ref{thm:ModelCommutative} it is natural to expect an equivalence
\[
	\operatorname{Alg}_{\mathcal O}(G\hyph\Sp_R)\simeq \Orb_{G,\mathcal O}\hyph\Com(\Sp)
\]
with diagrams indexed by the subcategory of the orbit category with morphisms corresponding to admissible subgroup inclusions (or rather subconjugacy classes).
However, it is more subtle to construct a comparison functor, which is being carried out in ongoing work.
\subsubsection*{Acknowledgements}
This article is a natural outgrowth and continuation of the work in my dotoral thesis written under the supervison of Stefan Schwede.
I would like to thank him as well as Markus Hausmann, Thomas Nikolaus, Irakli Patchkoria, and Emanuelle Dotto for helpful discussions.

\section{Preliminaries}

Let $G$ be a compact Lie group. At the point-set level we will work in the category $G\hyph\Top$
(resp.\ the based version $G\hyph\Top_\ast$) of compactly generated, weak Hausdorff spaces with continuous (based) $G$-action.
The set of based $G$-equivariant homotopy classes of maps between based $G$-spaces $X$, $Y$ will be denoted by $[X,Y]^G$.
We will often encounter categorical constructions (e.g.\ coproducts) indexed by conjugacy classes of subgroups (or morphisms),
where the individual terms make explicit reference to representatives. In that case we will simply write $\coprod_{(H\leq G)}X(H)$,
meaning that $H$ runs over a complete set of representatives that has been implicitly chosen.

We recall that the \emph{Weyl group} of a subgroup $H\leq G$ is the topological group $W_GH = N_GH/H$, 
where $N_GH = \{g\in G\mid H^g = H\}$ is the normalizer of $H$ in $G$. There is a canonical isomorphism
\begin{equation}\label{Eqn:WeylGroupAutomorphism}
W_GH \cong \Aut_G(G/H),\quad[g]\mapsto(G/H\xrightarrow{(-)\cdot g^{-1}}G/H)
\end{equation}
to the automorphism group of the $G$-space $G/H$. In particular, $W_GH$ acts on the $H$-fixed points of any $G$-space and
for the homogenous space $G/K$ there is a decomposition of $W_GH$-spaces
\begin{equation}\label{Eqn:WeylGroupDecomposition}
(G/K)^H \simeq\coprod_{\substack{(H'\leq G),\\H'\sim_G H}}W_GH/W_KH
\end{equation}
where the right-hand side is indexed by the $K$-conjugacy classes of subgroups that are $G$-conjugate to $H$.
This is a slight abuse of notation, since the $W_KH$ are not canonically embedded as subgroups of $W_GH$.
The map is classified by a choice of elements $g(H')\in G$ such that $H'= H^{g(H')}$, 
which also yields identifications $W_KH' \cong W_{^gK}H\leq W_GH$.

Let $\cc$ be a category and $I$ an (essentially small) indexing category.
The functor category $\Fun(I,\cc)$ of $I$-diagrams in $\cc$ will also be denoted by $\cc^I$ or $I\hyph\cc$
(similarly in the topologically enriched and higher categorical settings).
In particular, we write $\cc^{BG}$ for the category of $G$-objects in $\cc$, where $G$ is a (topological) group.
In the topologically enriched or higher categorical setting $BG$ means the full classifying space.

We will later work in the underlying $\infty$-categories of model categories (e.g.\ equivariant orthogonal spectra),
where the objects are essentially given by the objects of the 1-category.
To avoid confusion we will indicate in the notation that a categorical construction is homotopy invariant.
For example we write $X_{hG}$ for the homotopy orbits, which is the $\infty$-categorical quotient and would in general be denoted by $X/G$ for $X$ a $G$-object in some $\infty$-category. Instead, we reserve this notation to refer to the 1-categorical quotient
(which of course might represent the correct homotopy type).

\subsection{Equivariant orthogonal spectra}

We briefly review equivariant orthogonal spectra \cite{equivorth}, referring to \cite[Ch.\ 3]{global} for a detailed textbook exposition.
Let $\orth$ be the \emph{orthogonal indexing category}, i.e.\ the topological category with objects the finite dimensional inner product spaces
and morphism spaces given by the Thom spaces $\orth(V, W)$ of the orthogonal complement bundle
\[
	E(V,W)=\{(\phi, w)\in \Lin(V,W)\times W\mid w\in \phi(V)^{\perp}\}
\]
over the space of linear isometric embeddings $\Lin(V,W)$.
The composition is induced by the pairing
\[
	E(W,U)\times E(V,W)\rightarrow E(V,U), \quad ((\psi,u),(\phi, w))\mapsto (\psi\circ\phi, \psi(w)+u).
\]

\begin{Definition}
Let $G$ be a compact Lie group. The \emph{category of orthogonal $G$-spectra} $G\hyph\orthspec$
is the category of based continuous functors from $\orth$ to based $G$-spaces.
As a diagram category, it is canonically tensored, cotensored, and enriched over based $G$-spaces.
\end{Definition}

In particular, an orthogonal $G$-spectrum $X$ comes with \emph{structure maps} and \emph{action maps}
\[
	\sigma^X_{V,W}:X(V)\smashp S^W\rightarrow X(V\oplus W),\quad\alpha^X_{V,W}:\Lin(V, W)_+\smashp X(V)\rightarrow X(W)
\]
induced by the inclusion $S^W\subset \orth(V,V\oplus W)$ respectively by the zero section $\Lin(V, W)\rightarrow \orth(V, W)$. 
\begin{Remark}
It also turns out that this data is actually sufficient to uniquely describe $X$, cf.\ \cite[3.1.6]{global}.
\end{Remark}
In particular, the action maps restrict to an action of the orthogonal group $\orth(V)$ on the value $X(V)$ that commutes with the $G$-action.
If $V$ is an orthognal $G$-representation, we will always equip $X(V)$ with the resulting diagonal $G$-action.

\subsubsection{Equivariant homotopy groups}
The \emph{$0$-th equivariant homotopy group} of an orthogonal $G$-spectrum is defined as the directed colimit
\[
	\pi_0^G(X)=\colim_{V\subset\cu_G}[S^V, X(V)]^G
\]
over the subrepresentations of a complete $G$-universe $\cu_G$
(i.e.\ the countable direct sum of every $G$-representation embeds into $\cu_G$),
where the structure maps send a representative $f:S^V\rightarrow X(V)$ for $V\subseteq W$
to the composite (\cite[3.1.9]{global})
\[
	S^W\cong S^V\smashp S^{V^\perp}\xrightarrow{f\smashp \Id}X(V)\smashp S^{V^\perp}\overset{\sigma}{\longrightarrow} X(V\oplus V^\perp)\cong X(W).
\]
We will mostly be interested in the homotopy groups at finite subgroups $H\leq G$
where a countable sum $\cu_H=\oplus \rho_H$ of regular representations may be used.
The integer graded homotopy groups are obtained by looping or shifting the spectrum,
where for an orthogonal $G$-spectrum $X$ loops are applied levelwise to yield $\Omega X=\Omega X(-)$ and the \emph{shift} is given by $\operatorname{sh}X=X(\br\oplus -)$.
For $k\in\bz$ we then set
\[
	\pi_k^G(X)= \begin{cases}		\pi^G_0 \Omega^kX, &\text{if }k\geq 0\\
							\pi^G_0 \operatorname{sh}^{-k}X,&\mbox{if } k< 0\end{cases}
\]
and call a map $f:X\rightarrow Y$ a \emph{$\underline{\pi}_\ast$-isomorphism} or \emph{stable equivalence}
if it induces isomorphisms on $\pi_k^H(-)$ for all closed subgroups $H\leq G$ and $k\in\bz$.

\subsubsection{Products}
The category of orthogonal $G$-spectra carries a symmetric monoidal structure
\[
	-\smashp-:G\hyph\orthspec\times G\hyph\orthspec\rightarrow G\hyph\orthspec
\]
with unit the sphere spectrum $\bs$. The \emph{smash product} $X\smashp Y$ is determined up to preferred isomorphism as the initial example 
of a \emph{bimorphism} $b:(X,Y)\rightarrow Z$, i.e.\ a collection of maps $b_{V,W}:X(V)\smashp Y(Y)\rightarrow Z(V\oplus W)$ forming a natural transformation of functors
$\orth\times\orth\rightarrow \orth$. See \cite[3.5]{global} for details.

\begin{Definition}
We denote by $\Com(G\hyph\orthspec)$ the category of \emph{commutative orthogonal $G$-ring spectra}, that is,
the commutative monoid objects in orthogonal $G$-spectra with respect to the smash product.
\end{Definition}

\subsubsection{Model structures}\label{Sec:ModelStructures}
The $\underline\pi_\ast$-isomorphisms are part of a cofibrantly generated, topological, stable, and monoidal model structure on $G\hyph\orthspec$,
the (positive) \emph{flat} model structure of \cite{stolz}, where it is called the $\bs$-model structure.
We will also use the \emph{complete model structure} of \cite{kervaire},
which is defined for finite groups and has the same weak equivalences.
Let
\[
	\Sym X=\bigvee_{n\geq 0}\Sym^n X= \bigvee_{n\geq 0} X^{\smashp n}/\Sigma_n
\]
denote the symmetric algebra with respect to the smash product on orthogonal spectra,
so that $\Sym:\orthspec\rightarrow G\hyph\Com(G\hyph\orthspec)$ is the left adjoint of the forgeful functor $U$.
The model structures lift to commutative ring spectra such that weak equivalences and fibrations are detected on underlying spectra \cite[Thm.\  2.3.37]{stolz}, \cite[B.130]{kervaire}. In other words, the above forms a Quillen adjunction.
\begin{Remark}\label{Rmk:Error}
In order to show this, one needs to know that the symmetric powers $\Sym^n$ are homotopical on positively flat spectra.
However, in the case of the flat model structure there is a mistake going back to Mandell-May (\cite[Lemma III.8.4]{equivorth}, also see the discussion in \cite[B.120]{kervaire}).
It can be traced to \cite[Lemma 2.3.34]{stolz} and its proof where $\bp^{n}X$ is incorrectly identified with the 'naive' homotopy orbits
${E\Sigma_m}_+\smashp_{\Sigma_m}X^{\smashp m}$. Instead one has to use the 'genuine' version, where $E\Sigma_m$
is replaced by $E_G\Sigma_m$, a universal space for the family of \emph{graph subgroups} of $G\times \Sigma_m$
 (i.e.\ the graphs $\Gamma(\alpha)$ of homomorphisms $\alpha:H\rightarrow \Sigma_m$ from a closed subgroup of $H\leq G$).
The arguments of \cite{stolz} then go through to show that the projection to the point induces an equivalence
\begin{equation}\label{Eqn:GenuineOrbits}
	{E_G\Sigma_n}_+\smashp_{\Sigma_n}X^{\smashp n}\overset{\simeq}{\longrightarrow}X^{\smashp n}/\Sigma_n
\end{equation}
for positively flat $X$. The functor
\[
	{E_G\Sigma_n}_+\smashp_{\Sigma_n}(-):(G\times\Sigma_m)\hyph\orthspec\rightarrow G\hyph\orthspec
\]
is homotopical with respect to the graph equivalences ($\pi^{\Gamma}_\ast$-isomorphism for all graph subgroups $\Gamma$)
on $(G\times\Sigma_m)$-spectra, in contrast to ${E\Sigma_n}_+\smashp_{\Sigma_n}(-)$,
which preserves $G$-equivalences. So further homotopical analysis is required and one shows that $X^{\smashp m}$ is homotopical on positively flat spectra $X$.

Alternatively, one can also set up the model structure by following the strategy employed in \cite{global},
where a global model structure on commutative ring spectra is established by verifying a symmetrizability condition on generating (acyclic) cofibrations. Another possibility is to extend the complete model structure to compact Lie groups with equivalences the $\cf$-equivalences for a family $\cf$ of finite subgroups, which is all we will later need.

For finite groups $G$ all of this can be safely ignored and we just use the complete model structure in that case.
\end{Remark}

Since symmetric powers are left derivable in both model structures, an HHR-cofibrant replacement $X_c\rightarrow X$
of a pos.\ flat spectrum $X$ induces a weak equivalence
\[
	(X_c)^{\smashp n}/\Sigma_n\overset{\simeq}{\lra} X^{\smashp n}/\Sigma_n.
\]
\subsubsection{Geometric fixed point spectra}

We write $\rho_G=\br[G]$ for the regular representation of a finite group $G$ and equip it with the inner product such that the elements of $G$ form an orthonormal basis.
The norm element $N_G=\sum_{g\in G}g$ spans the canonical $\br$-summand of $G$-fixed points $(\rho_G)^G=\mathbb R \{N_G\}$.
\begin{Definition}
Let $G$ be a finite group and $X \in G\hyph\orthspec$ an orthogonal $G$-spectrum.
The \emph{geometric fixed point spectrum} $\Phi^G X \in \orthspec$ is defined at an inner product space $V$ as 
\[
(\Phi^G X)(V)=X(\regrep\tensor V)^G,
\]
where $G$ acts diagonally with respect to the action on $X$ and $\rho_G$.
The action maps are induced from those of $X$ and the structure maps 
for inner product spaces $V$ and $W$ are given by the composite
\begin{align*}
	X(\regrep\tensor V)^G\smashp S^W \cong (X(\regrep\tensor V)\smashp S^{\regrep\otimes W})^G\overset{\sigma^G}{\longrightarrow}&\quad X((\regrep\otimes V) \oplus(\regrep\otimes W))^G\\
	\cong&\quad X(\regrep\otimes (V\oplus W))^G,
\end{align*}
where the first isomorphism uses the preferred identification $(\regrep)^G=\br\{N_G\}\cong\br$.
\end{Definition}

If $G$ is a compact Lie group and $H\leq G$ a finite subgroup, we will write $\Phi^HX=\Phi^H(\res^G_HX)$
for the $H$-geometric fixed points of the restriction of the restriction to the subgroup $H$.
From the above definition we see that for based $G$-spaces $A$ and orthogonal $G$-spectra $X$
there is a natural isomorphism
\begin{equation*}
	\Phi^G(A\smashp X)\cong A^G\smashp\Phi^GX
\end{equation*}
and additionally for trivial $G$-spaces
\begin{equation*}
	\Phi^G(\map(A,X))\cong\map(A,\Phi^GX).
\end{equation*}
Moreover, cone sequence are preserved on the point-set level: If $f:X\rightarrow Y$ is a map of orthogonal $G$-spectra,
the canonical map $C(\Phi^G(f))\cong \Phi^G(C(f))$ from the mapping cone of $\Phi^G(f):\Phi^G(x)\rightarrow \Phi^G(Y)$
is an isomorphism.

\begin{Remark}\label{rmk:gfpmodels}
We can also express the above definition in a more diagrammatic fashion. Tensoring with the regular representation defines a topological functor
\[	\rho_G\otimes -:O\ra O_G	\]
from the orthogonal indexing category to the subcategory of $G$-representations and $G$-fixed morphism spaces
(i.e.\ pairs of equivariant isometries and points lying in the fixed points of the orthogonal complement).
On morphisms the functor $\rho_G\otimes (-)$ induces the map
\[
O(V,W)\ra O_G(\regrep\otimes V, \regrep\otimes W)=\left(O(\regrep\otimes V, \regrep\otimes W)\right)^G
\]
sending $[\phi, w]$ to $[\rho_G\otimes \phi, N_G\otimes w]$. 
The geometric fixed points of an orthogonal $G$-spectrum $X:O\ra G\hyph\Top_\ast$ are then given by the composite in the upper row of the diagram
\begin{center}
\begin{tikzcd}[column sep=large, row sep = small]
O	\arrow[r, "\rho_G\otimes -"]\arrow[rdd, "\cong Id"]	&	O_G		\arrow[rr, "X^G"]\arrow[dd, "(-)^G"]	&&	\Top_\ast,	\\	\\
								&	O		\arrow[rruu, "\Phi^G_M"]

\end{tikzcd}
\end{center}
where $X^G$ is the functor obtained by restricing to $O_G$ and taking fixed points at each value.
This also shows the relation to the construction discussed by Mandell-May in \cite[V.4]{equivorth},
which is referred to as the \emph{monoidal geometric fixed point functor} and denoted $\Phi^G_M$ in \cite[B.10]{kervaire}.
It is defined as the (topological) left Kan extension of $X^G$ along the functor $(-)^G:O_G\ra O$, $V\mapsto V^G$ and
the composition with $\rho_G\otimes -$ is canonically isomorphic to the identity.
The left Kan extension
comes with a natural transformation $X^G\Rightarrow \Phi^G_M\circ (-)^G$ and so we obtain a natural transformation
(now with respect to the spectrum $X$)
\[
\Phi^GX\ra \Phi^G_MX.
\]
As explained in \cite[B.10.5]{kervaire}, $\Phi^G_MX$ has the correct homotopy type if $X$ is cofibrant in the complete model structure.
In that case a comparison of homotopy groups shows that the above map is a $\pi_\ast$-isomorphism.
More precisely, one checks that the zig-zag of equivalences in \cite[B.201]{kervaire} identifies the induced map with the isomorphism of \cite[3.3.8]{global}.
\end{Remark}

We now discuss the functoriality of the geometric fixed points as the group $G$ varies.
Given a surjective homomorphism $\alpha:K\twoheadrightarrow G$, we define a $K$-equivariant linear isometric embedding of regular representations in the other direction as
\begin{align*}
\alpha_{!}:\alpha^\ast\regrep\hookrightarrow\regrep[K]\text{, }\quad g \mapsto \frac{1}{\sqrt{\vert\kernel\alpha\vert}}\sum_{k\in \alpha^{-1}(g)}k\text{.}
\end{align*}
This allows us to construct a natural \emph{inflation map}
\begin{equation}
\alpha^\ast:\Phi^GX\longrightarrow \Phi^K(\alpha^\ast X)
\end{equation}
between geometric fixed points of a $G$-orthogonal spectrum $X$. At an inner product space $V$ the restriction along $\alpha$ is the composite
\begin{align*}
X(\regrep\otimes V)^G=(\alpha^\ast X)((\alpha^\ast\regrep)\otimes V)^K\xrightarrow{(\alpha_!\otimes V)_\ast}&\quad (\alpha^\ast X)(\regrep[K]\otimes V)^K \text{.}
\end{align*}
Now let $\beta:L\twoheadrightarrow K$ be another surjective group homomorphism. By inspection of the formulas defining the embeddings of regular representations one sees that the diagram
\begin{center}
\begin{tikzcd}
(\alpha\beta)^\ast\regrep=\beta^\ast(\alpha^\ast\regrep ) \arrow[rr,"(\alpha\beta)_!"]\arrow[rd, "\beta^\ast\alpha_!"] 	&						&	\regrep[L]	\\
																& \beta^\ast\regrep[K]	\arrow[ru,"\beta_!"]
\end{tikzcd}
\end{center}
commutes and it follows that the restriction maps are compatible with composition:
\[
(\alpha\circ\beta)^\ast=\beta^\ast\circ\alpha^\ast:\Phi^GX\longrightarrow \Phi^L((\alpha\circ\beta)^\ast X).
\]

Let $G$ be a compact Lie group and $H\leq G$ a finite subgroup.
The \emph{conjugation map} associated with an element $g\in G$ is the composition
\begin{equation}
	g^\ast:\Phi^HX\xrightarrow{(c_g)^\ast}\Phi^{H^g}((c_g)^\ast X)\xrightarrow{(l^X_{g^{-1}})_\ast}\Phi^{H^g}X,
\end{equation}
where we write $H^g=g^{-1}Hg$ for the conjugated subgroup and $l^X_{g^{-1}}:(c_g)^\ast X\rightarrow X$ is the morphism given by left translation.
\begin{Definition}
Let $\preorbcatinv[G,\fin]$ denote the topological category of finite subgroups of $G$ with exactly one morphism $g:H\rightarrow H^g$ for every element $g\in G$.
It comes with a functor
\[
	\pi:\preorbcatinv[G,\fin]\rightarrow \orbcatinv[G,\fin], H \mapsto G/H
\]
to the subcategory of the orbit category of $G$ consisting of all isomorphisms.
\end{Definition}
\begin{Remark}\label{Rmk:DisjointUnionWeyl}
A choice of representatives for the conjugacy classes of finite subgroups of $G$ amounts to choosing a skeleton for $\orbcatinv[G,\fin]$ and gives an equivalence
\[
	\orbcatinv[G,\fin]\simeq \coprod_{(H\leq G)}B(W_GH)
\]
of topological categories with a disjoint union of Weyl groups, using their identification as automorphism groups (\ref{Eqn:WeylGroupAutomorphism}).
\end{Remark}
The preceeding discussion can be summarized as follows:
\begin{Proposition}\label{Prop:GfpsFunctoriality}
The conjugation maps endow the geometric fixed point construction with the structure of a continuous functor
\[
	\Phi:\preorbcatinv\times G\hyph\orthspec\rightarrow \orthspec.
\]
\end{Proposition}

We now consider the homotopy groups of geometric fixed point spectra.
Writing out the colimit defining them gives an identification
\begin{align*}
	\pi_0\Phi^GX=\colim_{n\geq 0}\pi_n X(\rho_G\otimes \mathbb R^n)^G
	&\cong \colim_{n\geq 0} [S^{(\rho_G\otimes \mathbb R^n)^G},X(\rho_G\otimes \mathbb R^n)^G]	\\
	&\cong \colim_{V\subset \cu_G} [S^{V^G}, X(V)^G]=\Phi^G_0 X,
\end{align*}
where the last isomorphism follows from the cofinality of the $\rho_G\tensor\br^n$ among orthogonal $G$-representations
and the equality on the right is the definition of the \emph{geometric fixed point homotopy groups} studied in \cite[3.3]{global}.
These also admit intrinsically defined inflation maps and by inspection the above isomorphism commutes with them.
Furthermore, there is a natural comparison map, the \emph{geometric fixed point map}
\begin{equation}\label{GeomFixMap}
	\pi_0^GX\longrightarrow \Phi^G_0X
\end{equation}
that sends a representative $f:S^V\rightarrow X(V)$ to $f^G:S^{V^G}\rightarrow X(V)^G$.
As with the equivariant homotopy groups, the definition of $\Phi_0$ is extended to all degrees by looping or shifting.
\begin{Proposition}
There is a natural isomorphism $\pi_\ast(\Phi X)\cong \underline\Phi_\ast (X)$ of graded $\Ho(\preorbcatinv)$-modules.
\end{Proposition}
\begin{proof}
The above identification in degree 0 prolongs to positive degrees, since geometric fixed point commute with looping.
There is a natural comparison map
\[
	\Phi^G(\sh X)\rightarrow \sh(\Phi^GX)
\]
induced by the linear isometry $(\rho_G\otimes V)\oplus\br\hookrightarrow \rho_G\otimes (V\oplus\br)$ including $\br$ as the $G$-fixed points.
We need to show that it is a stable equivalence.
There is a commutative square
\begin{center}
\begin{tikzcd}[row sep = large, column sep = large]
\sh(\Phi^GX)									&	\Phi^G(\sh X)		\arrow[l]	\\
\Sigma(\Phi^GX)	\arrow[u, "\lambda_{(\Phi^G X)}"]\arrow[r, "\cong"]		&	\Phi^G(\Sigma X)	\arrow[u, "\Phi^G(\lambda_X)"]
\end{tikzcd}
\end{center}
where for an orthogonal $G$-spectrum $Y$ the map $\lambda_Y:\Sigma Y\rightarrow \sh Y$ is defined levelwise as the composite
\[
	S^1\smashp Y(V)\cong Y(V)\smashp S^1\overset{\sigma}{\longrightarrow} Y(V\oplus\br)\cong Y(\br\oplus V).
\]
The left vertical map is a stable equivalence (\cite[III.1.25]{global}) and it follows that the comparison map between shifted spectra is surjective on homotopy. 
In the following we abbreviate notation and write $k+l\regrep$ for the representation $\br^k\oplus (\br^l\otimes \regrep)$. To see injectivity we start with an element $[f:S^{l+n}\ra X(1+ n\regrep)^G]\in\pi_l\Phi^G(\sh X)$ in the kernel and consider the diagram
\[
\begin{tikzcd}[row sep = small, column sep = small]
S^{l+n+1}\arrow[rr, "f\smashp S^1"]\arrow[d, draw = none, "\cong" description, sloped]	&&	X(1+n\regrep)^G\smashp S^1\arrow[r, "\sigma"]	\arrow[d, draw = none, "\cong" description, sloped]		&		X(1+n\regrep+1)^G\arrow[r]\arrow[d, draw = none, "\cong" description, sloped]	
	&	X(1+(n+1)\regrep)^G	\arrow[d, draw = none, "\cong" description, sloped]	\\
S^{l+n+1}\arrow[rr, "f\smashp S^1"]	&&	X(1+n\regrep)^G\smashp S^1\arrow[r, "\sigma"]	\arrow[rd]&	X(1+n\regrep+1)^G\arrow[r]							&	X((1+n)\regrep+1)^G		\\
&&&	X((1+n)\regrep)^G\smashp S^1\arrow[ru]	
\end{tikzcd}
\]
The two vertical isomorphisms on the right side are induced by the isometry that interchanges the outer $\br$ summands,
the other two are given by a degree $-1$ map in the right $\br$ coordinate. This makes the outer squares commute and the middle square commute up to homotopy.
The upper row is the stabilization of $f$ and thus represents the same class in the stable homotopy group.
By further stabilizing if necessary, we may assume that the composite
\[
	S^{l+n}\rightarrow X(1+n\rho_G)^G\rightarrow X((1+n)\rho_G)^G
\]
 and hence the lower row in the diagram is null-homotopic, showing that $[f]$ vanishes.
\end{proof}

\begin{Corollary}[{\cite[3.3.10, 3.3.5]{global}}]\label{Cor:OrbModules}
Let $\cf$ be a family of finite subgroups of a compact Lie group $G$. Geometric fixed points satisfy the following properties:
\begin{itemize}
	\item[(i)]
	A morphism $f:X\rightarrow Y$ of orthogonal $G$-spectra induces isomorphisms $\pi_\ast^H(X)\rightarrow \pi_\ast^H(X)$
	for all $H\in\cf$ if and only if it induces stable equivalences $\Phi^HX\rightarrow \Phi^HY$ for all $H\in \cf$.
	\item[(ii)]
	The homotopy groups $\pi_\ast (\Phi X)$ are $\Ho(\Orb_{G, \fin}^{\times})$-modules,
	that is inner conjugations act trivially.
\end{itemize}
\end{Corollary}

\subsection{Monoidal structure and norm maps}

\subsubsection{Lax monoidal structure}
We now equip the geometric fixed points construction with the structure of a lax symmetric monoidal functor.
We recall that the smash product $X\smashp Y$ of two orthogonal $G$-spectra $X$ and $Y$ comes with a universal bimorphism
$\iota:(X,Y)\rightarrow X\smashp Y$.
Applying $G$-fixed points yields a new natural bimorphism $(\Phi^GX,\Phi^GX)\longrightarrow\Phi^G(X\smashp Y)$
defined at the inner product spaces $V$ and $W$ by
\begin{align*}
	X(\regrep\otimes V)^G\smashp Y(\regrep\otimes W)^G \overset{\iota^G}{\lra}&\ (X\smashp Y)((\regrep\otimes V) \oplus (\regrep\otimes W))^G\\
	\cong&\ (X\smashp Y)(\regrep\otimes (V\oplus W))^G.
\end{align*}
The lax monoidal structure map is the associated natural transformation
\begin{equation}
	\mu^{\phi G}_{X,Y}:(\Phi^GX)\smashp(\Phi^GX)\longrightarrow\Phi^G(X\smashp Y).
\end{equation}
and the inclusion of fixed points \(	S^V\cong S^{(\regrep\otimes V)^G}	\)
defines the \emph{unit map} 
\[
	\eta^{\phi G}:\bs \ra \Phi^G\bs
\]
in each level $V$.
\begin{Proposition}
The maps $\eta^\phi$, $\mu^\phi$ constructed above define a lax symmetric monoidal structure on the geometric fixed point functor
\begin{align*}
\Phi: G\hyph\orthspec \lra \preorbcatinv\hyph\orthspec\text{,}
\end{align*}
where the diagram category $\preorbcatinv\hyph\orthspec$ is equipped with the object-wise symmetric monoidal structure.
\end{Proposition}
In fact, geometric fixed points are monoidal up to homotopy.
\begin{Proposition}\label{Prop:GfpMonoidal}
Let $X$ and $Y$ be cofibrant $G$-spectra in the complete model structure.
Then the composition 
\[
	\Phi^G(X)_c\smashp\Phi^G(Y)_c\rightarrow \Phi^G(X)\smashp\Phi^G(Y)\overset{\mu}{\lra}\Phi^G(X\smashp Y)
\]
is a $\pi_\ast$-isomorphism.
\end{Proposition}
\begin{proof}
This follows from the corresponding statement for the monoidal geometric fixed point functor \cite[B.199]{kervaire} and the observation that the comparison map of Remark \ref{rmk:gfpmodels} is a lax monoidal transformation.
\end{proof}
\subsubsection{Norm maps}

We have seen that geometric fixed points are lax symmetric monoidal and hence send commutative ring spectra to
to commutative ring spectra. But there is more structure available in the form of norm maps, which we will construct in this section.

We begin by recalling the $n$-fold \emph{wreath product} $\Sigma_n\wr G$ of a finite group $G$ for $n\in \bn$. This is the semi-direct product of the $n$-fold product $G^{\times n}$ and the symmetric group $\Sigma_n$ with respect to the right action of $\Sigma_n$ by permuting the factors. Concretely, elements are given by tuples $(\sigma ; g_1, \ldots, g_n)\in \Sigma_n\times G^{\times n}$ with multiplication
\[
(\sigma ; g_1, \ldots, g_n)\cdot(\tau ; k_1, \ldots, k_n)=(\sigma\circ\tau ; g_{\tau(1)}\cdot k_1, \ldots, g_{\tau(n)}\cdot k_n)\text{.}
\]
Wreath products naturally act from the left on coproducts of $G$-objects in any category and in the case of $G$-sets this gives an identification
\[
\Sigma_n\wr G\overset{\cong}{\lra} \Aut_G(G\sqcup\ldots\sqcup G)
\]
as the group of right $G$-equivariant automorphisms.

Next we recall that the \emph{norm} $N_H^GX$ of an $H$-spectrum $X$ is the 'indexed smash product' $\bigsmashp_{G/H}X$ with induced $G$-action \cite[2.2.3]{kervaire}.
To be more precise, we choose an (ordered) $H$-basis $b=(g_1,\ldots, g_m)$ of $G$, that is, a complete set of representatives for the right $H$-cosets in $G$. 
This is the same as the choice of a right $H$-equivariant isomorphism $\coprod_{i=1}^{m} H\overset{\cong}{\lra} G$.
As the $H$-automorphism group of $\sqcup H$, the wreath product $\Sigma_m\wr H$ acts freely and transitively from the right on the set of $H$-bases and this commutes with the $G$-action by left translation. So the above choice of an $H$-basis determines a group homomorphism $\Psi_b:G\ra \Sigma_m\wr H$ such that the square
\[
\begin{tikzcd}[column sep = large, row sep = large]
\coprod_{i=1}^m H	\arrow[rr, "{(g_1,\ldots, g_n)}"]\arrow[d, "\Psi_b(g)"]		&&	G	\arrow[d, "g\cdot"]	\\
\coprod_{i=1}^m H	\arrow[rr, "{(g_1,\ldots, g_n)}"]						&&	G
\end{tikzcd}
\]
commutes. In particular, $\Psi_b$ only depends up to congjugacy on $b$.
The norm is now defined as the $m$-fold smash power
\begin{equation}\label{eqn:norm}
	N_H^GX=\Psi_b^{\ast}(X^{\smashp m})
\end{equation}
with $G$-action obtained by restriction along the homomorphism $\Psi_b:G\ra \Sigma_m\wr H$.
The choice of coset representatives will usually be kept implicit in the notation.

We also recall that on commutative ring spectra, the norm functor
\[
	N_H^G:\Com(H\hyph\orthspec)\rightleftarrows\Com(G\hyph\orthspec):\res_H^G
\]
is left adjoint to restriction with counit defined by multiplying together the smash-factors.
\begin{Definition}\label{def:norms}
Let $H\leq G$ be a subgroup inclusion of index $m=(G:H)$. The \emph{norm map} is the natural transformation
\[
	N_H^G:\Phi^HX\lra\Phi^G(N_H^GX)
\]
defined for an $H$-orthogonal spectrum $X$ in level $V$ as the composition
\begin{align*}
	X(\regrep[H]\otimes V)^H\lra(X(\regrep[H]\otimes V)^{\smashp m})^{\Sigma_m\wr H}
	&\lra ((X^{\smashp m})((m\cdot\regrep[H])\otimes V))^{\Sigma_m\wr H}\\
	&\lra ((\Psi_b^\ast X^{\smashp m})(\regrep\otimes V))^G.
\end{align*}
The first map is the diagonal inclusion into the $m$-fold smash product, the second map is the iterated universal bimorphism
\[
	\iota_{W,\ldots,W}:X(W)\smashp\cdots\smashp X(W)\rightarrow (X^{\smashp m})(W\oplus\cdots\oplus W),
\]
and the last map is induced by the linear isometry $m\cdot\rho_H\cong \rho_G$ obtained by linearizing the choice of coset representative $\coprod_{i=1}^{m}H\cong G$ implicit in the definition of the norm.

For a commutative $G$-orthogonal ring spectrum $R$, we define a norm map
\[
	N_H^G:\Phi^HR\ra\Phi^G(N_H^G(\res^G_HR))\rightarrow \Phi^G R.
\]
as the composition with the counit of the the norm-restriction adjunction, which does not depends of the choice of cosets.
\end{Definition}

\begin{Proposition}
The norm maps satisfy the following relations, where the choice of coset representatives on the left hand side determines the one on the right hand side:
\begin{enumerate}
\item
They are transitive in subgroup inclusions. The equality
\[
N_K^G\circ N_H^K=N_H^G
\]
holds for all nested subgroup inclusions $H\leq K \leq G$.
\item
They commute with inflations. Let $\alpha:G\twoheadrightarrow K$ be a surjective group homomorphism and $H\leq K$ a subgroup.
We set $L=\alpha^{-1}(H)$, so that there is the commutative square
\[
	\begin{tikzcd}[column sep = large]
		G	\arrow[r, twoheadrightarrow, "\alpha"]													& 		K					\\
		L	\arrow[r, twoheadrightarrow, "\alpha_{\vert L}"]\arrow[u, sloped, draw=none, "\leq" description]		&		H	\arrow[u, sloped, draw=none, "\leq" description]	
	\end{tikzcd}
\]
Then the equality
\[
N_L^G\circ (\alpha_{\vert L})^\ast=\alpha^\ast\circ N_H^K
\]
holds.
\end{enumerate}
\end{Proposition}
\begin{proof}
Although it is lengthy and involves many commutative diagrams, the proof amounts to spelling out the definitions. We omit the details.
\end{proof}

\begin{Definition}\label{Def:PreOrbCat}
Let $G$ be a compact Lie group. We denote by $\preorbcat$ the topological category with objects the finite subgroups of $G$ and morphisms $g:H\rightarrow K$ those elements of $G$ such that $H^g\leq K$.
It comes with a functor
\[
	\pi:\preorbcat\rightarrow\orbcat,\quad H\mapsto G/H
\]
to the orbit category that sends a morphism $g:H\rightarrow K$ to
\[	
	G/H\overset{\cdot g^{-1}}{\longrightarrow}G/H^g\rightarrow G/K.
\]
\end{Definition}

\begin{Corollary}
The geometric fixed points construction defines a functor
\[
	\Phi:\Com(G\hyph\orthspec)\rightarrow \preorbcat\hyph\Com(\orthspec),
\]
where the subgroup inclusion $H\leq G$ encodes the associated norm map.
\end{Corollary}

\begin{Proposition}\label{Prop:GfpNorm}
Let $G$ be a finite group. For a cofibrant orthogonal $H$-spectrum $X$ in the HHR-model structure, the norm map
\[
	N_H^G:\Phi^HX\overset{\simeq}{\lra}\Phi^G(N_H^GX)
\]
is a $\pi_\ast$-isomorphism.
\end{Proposition}
\begin{proof}
This is \cite[B.209]{kervaire} combined with the identification at the end of Remark \ref{rmk:gfpmodels}.
\end{proof}


\subsection{Review of $\infty$-categories}
We will freely use the language of homotopical algebra and of $\infty$-categories as developed in terms of quasicategories in \cite{htt}.
Here we just review a few notions and fix notation.
The set of morphisms between two objects $x$ and $y$ in the homotopy category $\Ho(\cc)$ of a quasicategory $\cc$ will be denoted by $[x,y]_\cc$
and we write $\Map_\cc(x,y)$ for the corresponding mapping space. Since we do not need the details of a specific model it will suffice to think about this invariantly as a functor
$\Map_\cc(-,-):\cc^{\op}\times\cc\rightarrow \mathcal S$ into the $\infty$-category of spaces together with a natural identification $\pi_0\Map_\cc(x,y)\cong [x,y]_\cc$.
Additionally, for topological categories $\cc$ there is a functorial identification
\begin{equation}\label{eqn:TopologicalMappingSpaces}
	\Map_{{\Ncoh}\cc}(x,y)\simeq \Map_\cc(x,y)
\end{equation}
of the mapping space of the homotopy coherent nerve ${\Ncoh}\cc$ (or topological nerve, see \cite[1.1.5]{htt}) with the strict mapping space in $\cc$.
For two quasicategories $\cc$ and $\cd$, the functor $\infty$-category $\Fun(\cc,\cd)=\underline\Hom_{\operatorname{ sSet}}(\cc,\cd)$
is simply the internal hom in simplicial sets.
We recall that a colimit of a functor $F:I\rightarrow\cc$ is a universal \emph{cocone} on $F$,
i.e.\ a functor $\bar F:I^\triangleright\rightarrow\cc$ extending $F$, where $I^\triangleright=I\ast\Delta^0$ is the join with $\Delta^0$.
Following the usual abuse of notation we will simply write $\colim_I F$ for the value at the cone point
and suppress the additional data from the notation. 
An adjunction $F:\cc\rightleftarrows \cd:G$ will just consist of a pair of functors together with unit and counit transformations
satisfying the triangle equalities up to equivalence, c.f.\ \cite[5.2.2.8]{htt}.

\subsubsection{Kan extensions}
Let $f:\cc \rightarrow \cd$ be a functor and $\mathcal M$ an $\infty$-category with sufficiently many colimits,
i.e.\ those appearing in the formulas below. Then the restriction functor 
\[
	f^\ast:\Fun(\cd, \cm)\rightarrow \Fun(\cc,\cm)
\]
admits a left adjoint $f_!$ given by \emph{left Kan extension} \cite[4.3.3.7]{htt}.
It will later be important that the pointwise formula also holds in the $\infty$-categorical setting.
More precisely, the unit transformation $\eta:\Id \rightarrow f^\ast f_!$ exhibits the value
\begin{equation}\label{eqn:PointwiseFormula}
	(f_! X)(d)\simeq \colim_{\cc_{/d}}X
\end{equation}
at an object $d\in \cd$ as a colimit over the slice category $\cc_{/d}={\cc\times_{\cd}\cd_{/d}}$ \cite[4.3.2.2, 4.3.3.2]{htt}.
We will also need the following technical fact:
\begin{Lemma}\label{lemma:KanExtensionCocartesian}
Let $p:\cc\rightarrow \cd$ be a cocartesian fibration of quasicategories. 
Then the left Kan extension can be computed at an object $d\in \cd$ as the colimit
\[
	(p_!X)(d)\simeq \colim_{p^{-1}(d)}X
\]
over the fiber category $p^{-1}(d)$.
\end{Lemma}
\begin{proof}
Cocartesian fibrations are smooth \cite[4.1.2.15]{htt} and hence pullback (in simplicial sets) along them preserves cofinal functors \cite[4.1.2.10]{htt}.
This shows that for every object $d\in \cd$ the functor
\(
	p^{-1}(d)=\cc\times_\cd\{d\}\rightarrow \cc\times_\cd \cd_{/d}
\)
is cofinal.
\end{proof}

\subsubsection{Stability}
We recall that an $\infty$-category $\cc$ is \emph{pointed} if it has a zero object, i.e.\ an object that is both initial and terminal.
If morphisms in $\cc$ admit cofibers, the suspension functor $\Sigma:\cc\rightarrow \cc$ exists together with a natural equivalence
\[
	\Map_\cc(\Sigma X,Y)\simeq \Omega\Map_\cc(X,Y),
\]
which also uniquely determines it. Then $\cc$ is called \emph{stable} if it is pointed, admits finite colimits, and the suspension functor is an equivalence \cite[1.4.2.27]{ha}.
A functor between stable $\infty$-categories is \emph{exact} if it preserves zero objects and commutes with suspensions.

\begin{Lemma}\label{lemma:exactequivalence}
Let $F:\cc\rightarrow\cd$ be an exact functor between stable $\infty$-categories.
Then $F$ is an equivalence if and only if the induced functor $\Ho(F)$ on homotopy categories is an equivalence.
\end{Lemma}
\begin{proof}
Let $\Ho(F)$ be an equivalence. Then $F$ is also essentially surjective (by definition) and it remains to shows that it is fully faithful.
For every functor $F$ of pointed $\infty$-categories we have a commutative diagram
\begin{center}
	\begin{tikzcd}
		\pi_k\Map_\cc(x,y)\arrow[rr]\arrow[d, sloped, phantom, "{\cong}" description]	&&	\pi_k\Map_\cd(Fx,Fy)\arrow[d, sloped, phantom, "{\cong}" description]		\\
		\left[\Sigma^kx,y\right]\arrow[r]		&	\left[F(\Sigma^kx),Fy\right]	\arrow[r]	&	\left[\Sigma^kFx,Fy\right]
	\end{tikzcd}
\end{center}
in which the lower right map is induced by precomposition with $\Sigma^k(Fx)\rightarrow F(\Sigma^kx)$.
It follows that the upper map is an isomorphism if $F$ commutes with suspensions and the induced functor $\Ho(F)$ is fully faithful. 
If, in addition $\cc$ and $\cd$ are stable, then the mapping spaces are $H$-spaces and $\Map_\cc(x,y)\rightarrow\Map_{\cd}(Fx,Fy)$ 
induces isomorphisms on homotopy groups for all basepoints.
\end{proof}

The homotopy category $\Ho(\cc)$ of a stable $\infty$-category is canonically triangulated \cite[1.1.2.14]{ha}, with triangles represented by cofiber sequencs.
An object $C$ in a triangulated category with arbitrary sums is called \emph{compact}
if it corepresents a sum preserving functor $[C,-]_{\mathcal T}:\mathcal T\rightarrow \operatorname{Ab}$
and a set $\mathcal C$ of compact objects is a set of \emph{compact generators} if mapping out of $\cc$ detects zero objects.
We will make use of the following well-known fact from stable morita theory, whose proof is a standard localizing subcategory argument.

\begin{Proposition}\label{prop:eqstablecats}
Let $F:\cc\rightarrow \cd$ be an exact functor between stable $\infty$-categories that admit all sums.
Suppose that $F$ takes a set of compact generators $M$ to a set of compact generators
and that it is fully faithful when restricted to the full subcategory spanned by all suspensions
of objects in $M$. Then $F$ is an equivalence.
\end{Proposition}

\subsubsection{Localization}

Let $\mathcal C$ be a quasicategory and $W$ a collection of morphisms in $\mathcal C$.
The universal example of a functor $\cc\rightarrow \cd$ inverting $W$ will be denoted $\cc \rightarrow \cc[W^{-1}]$.
It is uniquely determined up to equivalence by the condition that for every $\infty$-category $\cd$ the precomposition
\[
	\Fun(\cc[W^{-1}],\cd)\rightarrow \Fun(\cc,\cd)
\]
is fully faithful with essential image consisting of those functors that send the maps in $W$ to equivalences \cite[1.3.4.1]{ha}.
It follows from the universal property that, given the top arrow $F$ in the diagram
\begin{center}
	\begin{tikzcd}
	\cc	\arrow[r, "F"]\arrow[d]				&	\cd	\arrow[d]	\\
	\cc[W^{-1}]	\arrow[r, dashed, "\tilde F"]	&	\cd[W'^{-1}]
	\end{tikzcd}
\end{center}
such that $F$ sends $W$ to a distinguished class of morphisms $W'$ in $\cd$,
there exists a unique functor on localizations together with a natural identification making the square commute.
\begin{Lemma}\label{Lemma:AdjunctionLocalization}
Let $F:\cc\rightleftarrows \cd:G$ be an adjunction such that $F$ and $G$ preserve the classes of morphisms $W$ and $W'$.
Then the induced functors on localizations also form an adjunction.
\end{Lemma}
\begin{proof}
Precomposition with the localization functor $\gamma_\cc$ induces and equivalence
\[
	\Nat(\Id, \tilde G\tilde F)\overset{\simeq}{\lra}\Nat(\gamma, \tilde G\tilde F \gamma)\simeq \Nat(\gamma, \gamma G F)
\]
between the spaces of natural transformations and so there is a unique arrow $\tilde\eta:\Id\ra\tilde G\tilde F$ corresponding to the composition $\gamma\eta$ with the unit of the adjunction on the right hand side. Analogously one defines a natural transformation $\tilde\xi:\tilde F\tilde G\ra \Id$ and checks that the triangle equalities are satisfied up to equivalence.
\end{proof}

An important special case is given by \emph{Bousfield localizations}, i.e.\ left adjoints $L$ to the inclusions of full subcategories $\cc^0\subseteq\cc$.
Equivalently, an endofunctor $L:\cc\rightarrow \cc$ together with a natural transformation $\eta:\Id\rightarrow L$ yields a Bousfield localization $\cc\rightarrow L\cc$
if and only if the two maps $L\eta, \eta L:L\rightarrow LL$ are equivalences \cite[5.2.7.4]{htt}.
We record some standard facts:
\begin{Proposition}\label{prop:BousfieldLocalization}
Let $L:\cc\ra\cc^0\subseteq \cc$ be a Bousfield localization.
    \begin{itemize}
    	\item[(a)]
	The localization functor $L:\cc\rightarrow \cc^0$ exhibits $\cc^0$ as the $\infty$-category obtained from $\cc$ by inverting the $L$-equivalences.
	\item[(b)]
	The $\infty$-category $\cc^0$ admits all colimits that exist in $\cc$ and they are computed by applying $L$ to the colimit formed in $\cc$.
	It also admits all limits that exist in $\cc$ and they are created in $\cc$.
	\item[(c)]
	Let $\cc$ be stable and $L:\cc\rightarrow\cc$ an exact endofunctor. Then $\cc^0$ is a stable subcategory of $\cc$.
    \end{itemize}
\end{Proposition}
\begin{proof}
Part (a) is \cite[5.2.7.12]{htt} and (b) follows formally as in the 1-categorical setting from adjointness and the fact that the restriction of $L$ to $\cc^0$
is naturally equivalent to the identity via the unit $\eta:\Id\rightarrow L$.

If $L$ is exact, then it follows that the suspension functor of $\cc$ restricts to that of $\cc^0$,
which is thus also an equivalence. This shows (c).
\end{proof}
\begin{Remark}
There is also the dual notion of Bousfield colocalization where the inclusion admits a right adjoint and the corresponding statements hold.
\end{Remark}

\subsubsection{Underlying $\infty$-categories of model categories}

We will mostly work with $\infty$-categories that are presented by model categories.
\begin{Definition}
Let $\mathcal M$ be a model category with weak equivalences $W$.
The \emph{underylying $\infty$-category} of $\mathcal M$ is the $\infty$-categorical localization
\(
	(N\mathcal M)[W^{-1}]
\)
of $\mathcal M$ at the weak equivalences.
\end{Definition}
In order to gain computational control over this,
we assume that $\cm$ is a topological model category (e.g.\ orthogonal spectra) and that $M$ admits functorial factorizations.
In that case the underlying $\infty$-category of $\cm$ can be identified with the homotopy coherent nerve $N_{\Delta}(\cm^\circ)$
of the full topological subcategory $\cm^\circ$ of bifibrant (i.e.\ fibrant-cofibrant) objects.
More precisely, the composition

\begin{equation}\label{eqn:TopologicalModelCatLocalization}
	N(\cm)\rightarrow N(\cm^\circ)\rightarrow N_{\Delta}(\cm^\circ)
\end{equation}
of a bifibrant replacement functor with the comparison map to the coherent nerve
exhibits the $N_{\Delta}(\cm^\circ)$ as the $\infty$-category obtained from $N\cm$ by inverting $W$ \cite[1.3.4.20, 1.3.4.16]{ha}.

Model categories represent nicely behaved $\infty$-categories:

\begin{Proposition}[{\cite[1.3.4.22, 1.3.4.24, 1.3.4.23]{ha}}]\label{Prop:ModelCategoryBicomplete}
Let $\mathcal M$ be a combinatorial model category.
\begin{itemize}
	\item[(i)]
	The underlying $\infty$-category of $\cm$ admits all (small) limits and colimits.
	\item[(ii)]
	Let $X:I\rightarrow M$ be a diagram in $\cm$ indexed by a small category $I$
	and $\alpha:\colim_{i\in I}X_i\ra m$ a map in $\cm$. Then $\alpha$ exhibits $m$ as a homotopy colimit of $X$
	if and only if
	\[
		N(I)^{\triangleright}\rightarrow N(\cm)\rightarrow N(\cm)[W^{-1}]
	\]
	is a colimit diagram in the underlying $\infty$-category of $\cm$.
	The analogous statement for limits also holds.
	\item[(iii)]
	The canonical functor
	\[
		N(\Fun(I,\cm))[W^{-1}]\overset{\simeq}{\longrightarrow}\Fun(N(I),N(\cm)[W^{-1}])
	\]
	is an equivalence.
\end{itemize}
\end{Proposition}
\begin{Remark}
In \cite{ha} Lurie works with the localization $N(\cm^c)[W^{-1}]$ of the category $M^c$ of cofibrant objects.
But this is equivalent to the localization of the full category $\cm$ (cf.\ \cite[1.3.4.16]{ha}).
An inverse to the inclusion is induced by a fibrant replacement functor.
\end{Remark}
\begin{Remark}
Let $\cm$ be a topological model category as above (not necessarily combinatorial).
Then it is still true that the underlying $\infty$-category admits colimits.
By \cite[4.4.2.6]{htt}, it is sufficient to show that coproducts and pushouts exist.
These are modeled by coproducts and homotopy pushouts on the pointset level,
which can be directly verified by inspection of mapping spaces.
Part (ii) also remains true for all other diagram shapes.
This follows from \cite[4.2.4.1]{htt} and using functorial factorizations.
\end{Remark}
We also review how Quillen adjunctions induce a derived adjunction on underlying $\infty$-categories,
which is slightly more complicated then the situation in Lemma \ref{Lemma:AdjunctionLocalization}.
Let $F:\cm\rightarrow \cn$ be a left Quillen functor. Then $F$ is \emph{left derivable} in the sense that
the dashed arrow in the square
\begin{center}
	\begin{tikzcd}
		N\cm\arrow[r, "F"]\arrow[d, "\gamma_{\cm}"]	&	N\cn\arrow[d, "\gamma_{\cn}"]		\\
		(N\cm)[W^{-1}]\arrow[r,dashed, "\mathbb LF"]\arrow[ru, Rightarrow, shorten <= 1em, shorten >=1em]			&	(N\cn)[W^{-1}]
	\end{tikzcd}
\end{center}
exists together with a natural transformation exhibiting it as an \emph{absolute right Kan extension} of $\gamma_{\cn} \circ F$ along $\gamma_{\cm}$,
i.e.\ for every functor $p:(N\cn)[W^{-1}]\rightarrow \cc$ the composite $p\circ\mathbb LF$ is a right Kan extension of $p\circ\gamma_{\cn} \circ F$.
Explicitely, $\mathbb LF$ is induced by the composition $F\circ (-)_c$ of $F$ with a cofibrant replacement functor,
cf.\ \cite[A.10]{nikolausscholze}.
Dually, a right Quillen functor $G:\cn\rightarrow \cm$ has a right derived functor $\mathbb RG$, which is an absolute left Kan extension.
\begin{Remark}\label{Rmk:QuillenBifunctor}
Quillen bifunctors $\cm\times\cm'\rightarrow \cm''$ can also be left derived by cofibrantly replacing in both factors.
In particular, for a monoidal model category $\cm$ there is a derived product functor 
\[
	\otimes:(N\cm)[W^{-1}]\times(N\cm)[W^{-1}]\rightarrow (N\cm)[W^{-1}].
\]
This can even be extended to a full symmetric monoidal structure on the underlying $\infty$-category (\cite[4.1.7.6]{ha}, \cite[Thm. A.7]{nikolausscholze}).
\end{Remark}
Now let $F$ and $G$ be adjoint to each other with unit $\eta$ and counit $\xi$. The derived unit transformation
\(
	\Id\rightarrow \mathbb RG\circ \mathbb LF
\)
corresponds under the universal property of $\mathbb LF$ to
\[
	\gamma_\cm\xrightarrow{\gamma_\cm\circ \eta}\gamma_\cm\circ G\circ F\rightarrow \mathbb RG \circ\gamma_\cn\circ F,
\]
where the last arrow is the natural transformation coming with $\mathbb RG$. Similarly, the derived counit
\(
	\mathbb LF\circ \mathbb RG\rightarrow \Id
\)
corresponds under the universal property of $\mathbb RG$ to
\[
	\mathbb LF\circ\gamma_\cm\circ G\rightarrow \gamma_\cn\circ F\circ G\xrightarrow{\gamma_\cn \circ \xi} \gamma_\cn.
\]
It follows by inspection that the triangle identities are satisfied up to equivalence and thus:

\begin{Proposition}\label{Prop:DerivedAdjunction}
Let $(F,G)$ be a Quillen pair. 
The above natural transformations yield a derived adjunction
\[
	\mathbb LF:(N\cm)[W^{-1}]\rightleftarrows(N\cn)[W^{-1}]:\mathbb RG.
\]
\end{Proposition}

We conclude by recording the expected preservation of stability under passage to underlying $\infty$-categories.
\begin{Proposition}\label{Prop:ModelCatStable}
Let $\cm$ be a stable model category. Then the underlying $\infty$-category of $\cm$ is also stable.
\end{Proposition}
\begin{proof}
We sketch the argument for a based topological model category such as orthogonal spectra.
Since $\cm$ is pointed so is the underyling $\infty$-category, which admits all colimits, in particular finite ones.
The suspension-loops adjunction
\[
	S^1\smashp-:\cm\rightleftarrows \cm:\Omega(-)
\]
is a Quillen equivalence and induces an equivalence on the underlying $\infty$-category,
where the suspension functor is modelled by the strict one.
\end{proof}

\section{The additive model for genuine $G$-spectra}

Let $G$ be a compact Lie group, $\mathcal F$ a family of closed (later always finite) subgroups of $G$, and $R\subseteq \mathbb Q$ a subring.
A morphism $f:X\rightarrow Y$ of orthogonal $G$-spectra is an \emph{$R$-local $\cf$-equivalence} if it induces an isomorphism
$(\pi_\ast^HX)\tensor R\rightarrow(\pi_\ast^HY)\tensor R$ for all $H\in \cf$ and we denote the collection of these by $W_{\cf, R}$.

\begin{Definition}
The $\infty$-category $\gensp$ of \emph{genuine $G$-spectra} is the underlying $\infty$-category
of orthogonal $G$-spectra, i.e.\ the localization at the $\underline\pi_\ast$-isomorphisms.
We denote by
\[
	\gensp_{R,\cf}=\N(G\hyph\orthspec)[(W_{\cf,R})^{-1}]\simeq (\gensp[G])[(W_{\cf,R})^{-1}]
\]
the $\infty$-category obtained by further inverting the $R$-local $\cf$-equivalences.
\end{Definition}

\begin{Remark}
This should not to be confused with the $\infty$-category $\Sp^{BG}$ of \emph{naive $G$-spectra},
which is the $\infty$-category of $G$-objects in spectra. It is also modelled by orthogonal $G$-spectra,
but with respect to the underlying weak equivalences, i.e.\  those maps which are $\pi_\ast$-isomorphisms of underlying non-equivariant spectra.
\end{Remark}
Let $\otimes$ denoted the derived smash product on genuine $G$-spectra,
which is objectwise given by $X\otimes Y\simeq X_c\smashp Y_c$ (cf.\ Remark \ref{Rmk:QuillenBifunctor}).
A map $f$ is contained in $W_{\cf, R}$ if and only if
\[
	E\cf_+\otimes X\otimes \bs_R \rightarrow E\cf_+\otimes Y\otimes \bs_R
\]
is a $\underline{\pi}_\ast$-isomorphism, where $E\cf$ is a universal space for the family $\cf$
and $\bs_R$ a Moore spectrum for the ring $R$ coming with a map
$\bs\rightarrow \bs_R$ that induces an isomorphism
$(\underline \pi_\ast \bs)\otimes R\cong \underline \pi_\ast \bs_R$.
\begin{Remark}
This is a Moore spectrum in the sense that it is the image of an ordinary Moore spectrum under the left adjoint
$\Sp\rightarrow \gensp$ to the forgetful functor. It can be constructed as a standard sequential colimit
\[
	\bs_R\simeq \colim (\bs\overset{x_1}{\lra} \bs \overset{x_2}{\lra}\bs \overset{x_3}{\lra}\cdots),
\]
where the sequence of $x_i\in \mathbb N$ contains every prime that is invertible in $R$ infinitely often.
\end{Remark}

We note that $\gensp_{R,\cf}$ is a smashing Bousfield (co-)localization of $\gensp$.
An orthogonal $G$-spectrum $X$ is called an \emph{$\cf$-spectrum} if the map $E\cf_+\otimes X\ra X$ is an equivalence.
It is \emph{$R$-local} if $X\rightarrow X\otimes \bs_R$ is an equivalence, that is $\underline{\pi}_\ast X\cong (\underline{\pi}_\ast X)\tensor R$.
The functors $E\cf_+\otimes - $ and $-\otimes \bs_R$, which commute with each other, are right- respectively left adjoints to the inclusions of the full subcategories of $\cf$-spectra or $R$-local spectra. 

\begin{Proposition}
The following holds for the $\infty$-category $\gensp_{\cf, R}$\emph{:}
\begin{itemize}
	\item[(i)]
	The composition $E\cf_+\otimes -\otimes \bs_R$ of Bousfield \emph(co-\emph)localization functors identifies
	$\gensp_{R,\cf}$ with the full subcategory of $\gensp$ spanned by the $R$-local $\cf$-spectra. In particular,
	there is an identification
	\[
		\Map_{(\gensp)_{\cf,R}}(X,Y)\simeq \Map_{\gensp}(E\cf_+\otimes X, Y\otimes \bs_R)
	\]
	of mapping spaces.
	\item[(ii)]
	It is stable and admits all \emph(small\emph) limits and colimits.	
	\item[(iii)]
	The $R$-local homotopy groups 
	\[
		[\suspsp_+(G/H), X]_{\gensp_{\cf,R}}\cong(\pi^H_0X)\tensor R
	\]
	are corepresented in the homotopy category by the suspension spectra of transitive $G$-sets. In particular,
	$\{\suspsp_+(G/H)\}_{H\in \cf}$ forms a set of compact generators. 
\end{itemize}
\end{Proposition}
\begin{proof}
Part (i) follows Proposition \ref{prop:BousfieldLocalization}.(a).
As the underlying $\infty$-category of a stable model category
$\gensp$ is stable and bicomplete (Propositions \ref{Prop:ModelCatStable} and \ref{Prop:ModelCategoryBicomplete}) and this also follows for the (exact) Bousfield (co-)localization $\gensp_{\cf, R}$ (Proposition \ref{prop:BousfieldLocalization}), showing (ii).
For part (iii) we remark that corepresentability of homotopy groups in $\gensp$ is well-known, but not well-documented in the literature. It suffices to check on fibrant spectra in the projective model structure of \cite{equivorth} (the $G$-$\Omega$ spectra) and in that case it follows by inspection.
We observe that the projection $E\cf\times G/H\rightarrow G/H$ is a weak $G$-equivalence for $H\in\cf$.
Hence $\suspsp_+(G/H)$ is an $\cf$-spectrum and we obtain isomorphisms
\[
	[\suspsp_+(G/H), X]_{\gensp_{\cf, R}}\cong [\suspsp_+(G/H), X\otimes \bs_R]_{\gensp}\cong(\pi^H_0X)\tensor R.
\]
These are sum-preserving functors in $X$ and as $H$ varies over the subgroups of $\cf$ they detect equivalences by definition, showing that the corepresenting objects form a set of compact generators.
\end{proof}

\subsection{The comparison functor}
We now describe the passage of geometric fixed points to a functor on underlying $\infty$-categories.
From now on $\cf$ is a family of \emph{finite} subgroups and we denote by $\preorbcatinv[\cf]\subset\preorbcatinv$ the full subcategory spanned by the subgroups $H\in \cf$.
Since $\Phi^H$ is homotopical (Corollary \ref{Cor:OrbModules}), the equivalence
\[
	\Phi^H(E\cf_+\otimes X\otimes \bs_R)\simeq (E\cf_+)^H\otimes (\Phi^HX)\otimes (\Phi^H\bs_R)\simeq(\Phi^HX)\otimes \bs_R
\]
shows that the functor
\(
	\Phi:G\hyph\orthspec\ra \Fun^{\cts}(\preorbcatinv[\cf], \orthspec)
\)
sends $R$-local $\cf$-equivalences to levelwise $R$-local equivalences.
Hence the composition with 
\begin{align*}
	N(\Fun^{\cts}(\preorbcatinv[\cf], \orthspec))\ra N_{\Delta}(\Fun^{\cts}(\preorbcatinv[\cf], \orthspec))
	&\ra\Fun(N_{\Delta}\preorbcatinv[\cf], N_{\Delta}\orthspec)	\\
	&\ra\Fun(N_{\Delta}\preorbcatinv[\cf], (N_{\Delta}\orthspec)[W_R^{-1}])	
\end{align*}
factors over a unique functor 
\(
	\gensp_{\cf, R}\rightarrow \preorbcatinv[\cf]\hyph\Sp_R
\)
on localizations, where we no longer distinguish between a topological category and its coherent nerve in the notation for the indexing category.
\begin{Definition}\label{Def:GeometricFixedPointsInfinity}
The \emph{induced geometric fixed point functor} on underlying $\infty$-categories is defined as the composition
\begin{equation*}
	\Phi:\gensp_{\cf, R}\rightarrow \preorbcatinv[\cf]\hyph\Sp_R\xrightarrow{\pi_!}\orbcatinv[\cf]\hyph\Sp_R,
\end{equation*}
where $\pi_!$ is the left Kan extension along the projection $\pi:\preorbcatinv[\cf]\rightarrow \orbcatinv[\cf]$ (see Definition \ref{Def:PreOrbCat}).
\end{Definition}

\begin{Proposition}
Geometric fixed points induce an exact and sum-preserving functor of stable $\infty$-categories such that the universal transformation
$\eta:\Id\rightarrow \pi^\ast\circ \pi_!$ makes the diagram of $\infty$-categories
\begin{center}
	\begin{tikzcd}[row sep = large, column sep = large]
		\gensp_{R,\cf}\arrow[r,"\Phi"]\arrow[rd, "\Phi^H"] 		&	\orbcatinv[\cf]\hyph\Sp_R	\arrow[d, "\ev_{G/H}"]	\\
													&	\Sp_R	
	\end{tikzcd}
\end{center}
commute for all $H\in \cf$.
\end{Proposition}
\begin{proof}
Lemma \ref{Lemma:KanExtensionOrbit} also holds for the invertible orbit category and hence $\eta$ factors over an equivalence
\[
	(\Phi^HX)_{hH}\overset{\simeq}{\longrightarrow}(\pi_!\Phi X)(G/H).
\]
For every $R$-local $Y\in \Sp_R^{BH}$ the projection $Y\rightarrow Y_{hH}$ induces an isomorphism $(\pi_\ast Y)/H\overset{\cong}{\longrightarrow}\pi_\ast(Y_{hH})$ since the order of $H$ is invertible.
We already know that algebraically the $H$-action is trivial for $Y=\Phi^HX$ (Corollary \ref{Cor:OrbModules}) and so the projection is an equivalence in that case.

Exactness of $\Phi$ is detected pointwise and hence follows from exactness of the functors $\Phi^H$,
which in turn follows because they preserve cone sequences on the pointset level.

\end{proof}

\begin{Remark}
The fact that the geometric fixed points a priori only yield $\preorbcatinv$-diagrams is a technical artifact of the pointset model
and the use of left Kan extensions is a way of correcting that 'defect'.
This should be regarded as separate from the later use of Kan extensions in Section \ref{Sec:MultEq},
where they form a conceptual part of the argument. The author is not aware of another lax monoidal construction that is either homotopical or can be derived on commutative ring spectra.
\end{Remark}

\subsection{The additive equivalence}
We now assume that the group orders of elements in $\cf$ are invertible in $R$.
In order to show that the functor $\Phi$ from Definition \ref{Def:GeometricFixedPointsInfinity} is an equivalence, we need the following computational input:
\begin{Lemma}\label{Lemma:FixedPointCorep}
For finite subgroups $H$ and $K$ of $G$ there is a natural isomorphism
\[
	[\suspsp_+(G/K)^H, Y]_{B(W_GH)\hyph\Sp_R}\cong\bigoplus_{(\hat H\leq K), \hat H \sim_G H}(\pi_0Y)^{W_K\hat H},
\]
where the sum is indexed by $K$-conjugacy classes of subgroups of $K$ that are conjugate to $H$ in $G$.

\end{Lemma}
\begin{proof}
We recall the decomposition formula (\ref{Eqn:WeylGroupDecomposition})
\[
	(G/K)^H\cong \coprod_{(\hat H\leq K), \hat H \sim_G H} W_GH/W_K{\hat H}.
\]
The statement then from follows from the general fact that the coset $G/H$ represents $H$-homotopy fixed points in $\Fun(BG, \Sp)$
and that for invertible group orders their homotopy groups are obtained by taking algebraic fixed points:
\[
	[\suspsp_+(G/H),X]_{\Fun(BG, \Sp)}\cong[\bs, X]_{\Fun(BH, \Sp)}\cong (\pi_0X)^H
\]
\end{proof}

\begin{Corollary}
The geometric fixed points $\{{\Phi(\suspsp_+(G/K))}\}_{K\in\mathcal F}$ form a set of compact generators for $\Orb_{\mathcal F}^{\times}\hyph\Sp_R$.
\end{Corollary}
\begin{proof}
Under the equivalence (cf.\ Remark \ref{Rmk:DisjointUnionWeyl})
\[
\prod_{(H\in \cf)}\ev_{G/H}:\orbcatinv[\cf]\hyph\Sp_R\overset{\simeq}{\longrightarrow}\prod_{(H\in \cf)}B(W_GH)\hyph\Sp_R
\]
the diagram $\Phi(\suspsp_+(G/K))$ is mapped to $(\suspsp_+(G/K)^H)_{(H\in\cf)}$.
By the previous lemma, $\suspsp_+(G/K)^H$ corepresents a sum-preserving functor that contains $\pi_0(-)$ as a retract.
Hence it is compact and a generator for $B(W_GH)\hyph\Sp_R$.
\end{proof}

We now come to the main result of this section, describing $R$-local $\cf$-spectra in terms of geometric fixed points.

\begin{Theorem}\label{Theorem:AdditiveEquivalence}
Let $G$ be a compact Lie group, $\mathcal F$ a family of finite subgroups of $G$ and $R\subseteq \bbq$ a ring
such that the group orders of $\mathcal F$ are invertible in $R$. Then the geometric fixed point functor $\Phi$ induces an equivalence
of $\infty$-categories
\[
	\gensp_{\mathcal F, R}\simeq \orbcatinv[\cf]\hyph \Sp_R.
\]
\end{Theorem}
\begin{proof}

For every orthogonal $G$-spectrum $X$, the components of the induced map
\begin{align*}
	\Phi:[\suspsp_+(G/K),X]_{G\hyph\Sp}	&\lra[\Phi (\suspsp_+(G/K)), \Phi X]_{\orbcatinv[\cf]\hyph\Sp}	\\
	&\cong \prod_{(H\leq G)}[\suspsp_+(G/K)^H, \Phi^H X]_{B(W_GH)\hyph\Sp}
\end{align*}
can be identified via Lemma \ref{Lemma:FixedPointCorep} with the geometric fixed point map (\ref{GeomFixMap})
\[
	\pi_0^K(X)\lra \prod_{(\hat H\leq K), \hat H \sim_G H}(\Phi_0^{\hat H}X)^{W_K\hat H}.
\]
As $H\in \mathcal F$ runs over the $G$-conjugacy classes in $\mathcal F$, the $\hat H$ hit every $K$-conjugacy class of $\mathcal F$ exactly once.
Hence the product of these maps is an isomorphism by \cite[3.4.28]{global}.
In particular, $\Phi$ is fully faithful on a set of compact generators and by the corollary these are also mapped to a set of compact generators.
Proposition \ref{prop:eqstablecats} then implies that $\Phi$ is an equivalence.
\end{proof}

\section{Commutative ring spectra}
We recall from Section \ref{Sec:ModelStructures} that the category of commutative $G$-ring spectra $\Com(G\hyph\orthspec)$ admits a transferred model structure,
that is weak equivalences and fibrations are detected by the forgetful functor $U:\Com(G\hyph\orthspec)\rightarrow G\hyph\orthspec$.
The $\infty$-category of \emph{genuine commutative $G$-ring spectra} 
\(
	\Com(\gensp)
\)
is the underlying $\infty$-category and $\Com(\gensp)_{\cf, R}$ denotes the localization at the $R$-local $\cf$-equivalences.

\begin{Remark}
Inverting the $\cf$-equivalences on commutative ring spectra can no longer be expressed via the right Bousfield localization functor
$E\cf_+\otimes -$ because there is no map $\Sigma^{\infty}_+E\cf \rightarrow \bs$ of commutative ring spectra (unless $\cf$ containes all subgroups).
\end{Remark}

Instead, we use the \emph{$\cf$-completion} $F(E\cf_+,-)$, which is a left Bousfield localization functor.
It restricts to an endofunctor of $\Com(\gensp)$ since it is modeled by a lax symmetric monoidal functor that is also right derivable on commutative ring spectra (the model structure is topological).
Commutative ring spectra can still be $R$-localized because the Moore spectrum $\bs_R$ admits the structure of a commutative ring spectrum, although this is a non-trivial fact (Proposition \ref{Prop:MooreSpectrumCommutative}).
Put together, we observe that $F(E\cf_+,-\otimes\bs_R)$ is a Bousfield localization identifying $\Com(\gensp)_{R,\cf}$ with the full subcategory of $R$-local $\cf$-complete commuative ring spectra.
In particular, it admits (small) colimits and limits, the latter of which are preserved by the forgetful functor
$\mathbb U:\Com(\gensp)_{R,\cf}\rightarrow \gensp_{R,\cf}$, which is a right adjoint (Proposition \ref{Prop:MonadicAdjunction}).
\subsection{The comparison functor}

Analogously to the additive case, geometric fixed points induce a functor
\begin{equation}
	\Phi_{\Com}:\Com(\gensp)_{\cf,R}\lra \preorbcat[\cf]\hyph\Com(\Sp)_R\overset{\pi_!}{\longrightarrow}\orbcat[\cf]\hyph\Com(\Sp)_R
\end{equation}
on underlying $\infty$-categories of commutative ring spectra, where $\pi$ is the projection from Definition \ref{Def:PreOrbCat}. To see that this 'forgets' to the functor of the previous section, we need the following lemmata.

\begin{Lemma}\label{Lemma:HomotopyOrbitsRingSpectra}
Let $K$ be a finite group and $S\in \Com(\Sp)^{BK}$ a commutative ring spectrum with $K$-action that is trivial on $\pi_\ast S$.
Then the projection map
\(
	S\rightarrow S_{hK}
\)
to the quotient in commutative rings is a $\bz[\frac{1}{\lvert K \rvert}]$-local equivalence.
\end{Lemma}
\begin{proof}
Homotopy fixed points are computed in the underlying category of spectra and so it follows by inspection that the canonical map
\[
	(S^{hK})_{\tr}\rightarrow S
\]
is an equivalence for $\bz[\frac{1}{\lvert K \rvert}]$-local $S$ with trivial action on homotopy groups.
We may thus assume that $S$ itself has trivial action and compute the functor corepresented by $S^{hK}$ in the homotopy category
as follows:
\[
	[S_{hK}, T]_{\Com(\Sp)}\cong[S, T]_{\Com(\Sp)^{BK}}\cong[S, T^{hK}]_{\Com(\Sp)}\cong[S, T]_{\Com(\Sp)}.
\]
This isomorphism is given by precomposition with the projection $S\rightarrow S_{hK}$,
which thus is an equivalence.
\end{proof}

\begin{Lemma}\label{Lemma:KanExtensionOrbit}
Let $\cc$ be an $\infty$-category with sufficiently many colimits.
The left Kan extension of a diagram $X:\preorbcat[\cf]\rightarrow \cc$ along the functor
\(
	\pi:\preorbcat[\cf]\rightarrow \orbcat[\cf]
\)
can be computed at $G/H$ as the quotient
\[
	(\pi_!X)(G/H)\simeq X(H)/H.
\]
\end{Lemma}
\begin{proof}
On the level of topological categories the functor $\pi$ induces Serre-fibrations on mapping spaces,
since for each $H\in\cf$ the induced map
\[
	\preorbcat[\cf](-,H)\rightarrow \orbcat[\cf](\pi(-),G/H)\cong \preorbcat[\cf](-,H)/H
\]
exhibits the target as the quotient of a free $H$-action.
This also shows that for every morphism $H\rightarrow K$ in $\preorbcat[\cf]$ the square
\begin{center}
	\begin{tikzcd}
	\preorbcat[\cf](K,L)\arrow[r]\arrow[d]			&	\preorbcat[\cf](H,L)\arrow[d]	\\
	\orbcat[\cf](G/K,G/L)\arrow[r]				&	\orbcat[\cf](G/H,G/L)
	\end{tikzcd}
\end{center}
is a (homotopy) pullback and hence every morphism is cocartesian by \cite[2.4.1.10]{htt}. 
By Lemma \ref{lemma:KanExtensionCocartesian} the Kan extension can be computed pointwise as the colimit over the fibers $\pi^{-1}(G/H)\simeq BH$.
\end{proof}

\begin{Proposition}
Geometric fixed points induce a functor $\Phi^{\Com}$
on commutative ring spectra such that the square
\begin{center}
	\begin{tikzcd}
	\gensp_{\mathcal F, R}	\arrow[r, "\Phi"]	&	\Orb^{\times}_{\mathcal F}\hyph \Sp_R		\\
	\Com(\gensp)_{\mathcal F, R}	\arrow[r, "\Phi^{\Com}"]\arrow[u, "\mathbb U"]	&	\Orb_{\mathcal F}\hyph \Com(\Sp)_R	\arrow[u, "\mathbb U"]
	\end{tikzcd}
\end{center}
commutes.
\end{Proposition}
\begin{proof}
We consider the diagram
\begin{center}
	\begin{tikzcd}
	\gensp_{\mathcal F, R}	\arrow[r]		& 	\preorbcatinv[\mathcal F]\hyph \Sp_R\arrow[r, "\pi_!"]		&\Orb^{\times}_{\mathcal F}\hyph \Sp_R		\\
	\Com(\gensp)_{\mathcal F, R}	\arrow[r]\arrow[u, "\mathbb U"]	&	\preorbcat[\cf]\hyph\Com(\Sp)_R\arrow[r, "\pi_!"]\arrow[u, "\mathbb U"]	&
	\Orb_{\mathcal F}\hyph \Com(\Sp)_R	\arrow[u, "\mathbb U"]
	\end{tikzcd}
\end{center}
in which the left square commutes strictly on the pointset level. Unit and counit of the left Kan extension induce a natural transformation 
\[
	\pi_! \bu\ra\pi_!\bu(\pi^\ast\pi_!)=(\pi_!\pi^\ast)\bu\pi_!\ra\bu\pi_!
\]
between the two composites in the right square, which is an equivalence by the pointwise description of Kan extensions and Lemma \ref{Lemma:HomotopyOrbitsRingSpectra}.
\end{proof}

\subsection{The multiplicative equivalence}\label{Sec:MultEq}
In order to show that the comparison functor $\Phi^{\Com}$ is an equivalence, we need to analyse the free-forgetful adjunctions.
Let $\mathbb P:\gensp\rightarrow \Com(\gensp)$ be the derived symmetric algebra functor,
which is left adjoint to the forgetful functor $\mathbb U$.
The underlying homotopy type of $\mathbb PX$ is given by the sum of derived symmetric powers
\(
	\mathbb U( \mathbb PX)\simeq \bigoplus_{m\geq 0}\mathbb P^m X.
\)
For a positive flat orthogonal $G$-spectrum $X$ the projection map 
\[
	{E_G\Sigma_m}_+\wedge_{\Sigma_m}X^{\wedge m}\overset{\simeq}{\lra} X^{\wedge n}/\Sigma_n
\]
is a $\underline\pi_\ast$-isomorphism and this yields an identification with the equivariant extended powers
\[
	\mathbb P^m \simeq {E_G\Sigma_m}_+\wedge_{\Sigma_m}(-)^{\otimes m}.
\]

\begin{Lemma}
The functor $\mathbb P$ preserves $R$-local $\cf$-equivalences.
\end{Lemma}
\begin{proof}
Since ${E_G\Sigma_m}_+\wedge_{\Sigma_m}(-)$ commutes with $A\otimes (-)$ for $\Sigma_n$-trivial $A$,
this follows from the $(G\times \Sigma_m)$-equivalences $\Delta:E\cf\overset{\simeq}{\longrightarrow} (E\cf)^{\times m}$
and $\mathbb S_{\bbq}\simeq \mathbb (S_{\bbq})^{\otimes m}$. The latter is induced by the composition
\(
\mathbb S \simeq \mathbb S^{\otimes m}\rightarrow (\mathbb S_{\bbq})^{\otimes m},
\)
where the first identification comes from the unit isomorphism on the pointset level.
Alternatively, the statement also follows from the computation in the proof of Proposition \ref{Prop:FreeCommutative}.
\end{proof}

\begin{Proposition}\label{Prop:MooreSpectrumCommutative}
The Moore spectrum $\bs_R$ can be constructed as a commutative $G$-ring spectrum.
\end{Proposition}
\begin{proof}
We implement the sketch of \cite{EquivClosure} in the $\infty$-category of genuine commutative $G$-rings.
Let $A=\colim_{n\in\bn} A_n$ be the colimit (starting with $A_0=\bs$) over the lower maps in the pushouts
\begin{center}
	\begin{tikzcd}
	\bp(Z_n)\arrow[r]\arrow[d]	&	\bp(\ast)\simeq\bs\arrow[d]	\\
	A_n		\arrow[r]		&	A_{n+1},
	\end{tikzcd}
\end{center}
where $Z_n=\bigoplus_{\alpha}\bs/l(\alpha)$ is a sum of mod $l$ Moore spectra indexed by the homotopy classes of maps
$\alpha:\bs/l(\alpha)\rightarrow A_n$ such that $l(\alpha)$ is invertible in $R$.
By construction, for any mod l Moore spectrum such that $l$ is invertible in $R$ the induced map
\[
	[\bs/l,A_n]\rightarrow [\bs/l,A_{n+1}]
\]
on homotopy classes is trivial. It follows that $A$ is $R$-local, since sequential colimits of commutative ring spectra are computed in
the underlying category of spectra (this follows from the corresponding pointset level statement because cofibrations of commutative ring spectra are in particular h-cofibrations of underlying spectra).
By the previous lemma, the upper horizontal maps are $R$-equivalences. Since the pushout in commutative ring spectra is given by the derived relative smash product $A_{n+1}\simeq A_n\otimes_{\bp(Z_n)}\bs$, the lower maps are also $R$-equivalences and hence so is the composition $\bs \rightarrow A\simeq \bs_R$.
\end{proof}

\begin{Proposition}\label{Prop:MonadicAdjunction}
The Quillen adjunction $\Sym:G\hyph\orthspec\rightleftarrows \Com(G\hyph\orthspec):U$ induces a monadic adjunction
\[
	\mathbb P:\gensp_{\cf, R}\rightleftarrows \Com(\gensp)_{\cf, R}:\mathbb U
\]
on underlying $\infty$-categories.
\end{Proposition}
\begin{proof}
By the previous lemma the derived adjunction
\(
	\mathbb P:\gensp\rightleftarrows \Com(\gensp):\mathbb U
\)
passes to an adjunction on localizations (cf.\ Lemma \ref{Lemma:AdjunctionLocalization}).
It is monadic because $\bu$ detects equivalences (by definition) and preserves homotopy colimits of simplicial objects (see the appendix).
\end{proof}

The left adjoint of the forgetful functor $\bu:\orbcat[\cf]\hyph\Com(\gensp)_R\rightarrow\orbcatinv[\cf]\hyph\Sp_R$
is given by the composition
\begin{equation}
	\orbcatinv[\cf]\hyph\Sp_R\xrightarrow{\iota_!}\orbcat[\cf]\hyph\Sp_R\xrightarrow{\bp}\orbcat[\cf]\hyph\Com(\Sp)_R
\end{equation}
of left Kan extension along $\iota:\orbcatinv[\cf]\rightarrow\orbcat[\cf]$ with the objectwise free commutative algebra functor.

\begin{Remark}
The adjunction is also monadic because $\iota$ is essentially surjective and colimits in diagram categories are computed pointwise.
\end{Remark}
We will need a more concrete description of the left adjoint:
\begin{Lemma}\label{Lemma:FreeCommutative}
Let $\mathcal C$ be an $\infty$-category that admits sufficently many colimits and $X:\Orb^{\times}_{\mathcal F}\rightarrow \mathcal C$ a functor.
\begin{itemize}
\item [(i)]
The left Kan extension $\iota_!X:\Orb_{\mathcal F}\rightarrow \mathcal C$ of $X$ to the full $\cf$-orbit category evaluated at $G/K$ can be computed by the formula
\[
	(\iota_! X)(G/K)\simeq \coprod_{(H\leq K)}X(G/H)/{W_KH}.
\]
On the summand indexed by $H$ the map is induced by the composition
\[
	X(G/H)\overset{\eta}{\longrightarrow}(\iota_! X)(G/H)\xrightarrow{\tr_H^K}(\iota_! X)(G/K),
\]
where $\eta:X\rightarrow \iota^\ast\iota_! X$ is the universal transformation.
\item [(ii)]
Let $\mathcal C$ be symmetric monoidal and suppose that the product $\otimes$ preserves coproducts in each variable.
Then the free commutative algebra is described by the following formula:
\[
	\mathbb P((\iota_! X)(G/K))\simeq\coprod_{\substack{(\alpha:K\rightarrow \Sigma_m),\\ m\geq 0}}\left(\bigotimes_{[j]\in K\backslash\{1,\ldots,m\}}X(G/\Stab_K j)\right)/C(\alpha).
\]
Here the identification is given on each summand by the product of the maps
\[
	X(G/\Stab_K j)\longrightarrow \mathbb P((\iota_! X)(G/\Stab_K j)) \xrightarrow{\bp\tr} \mathbb P((\iota_! X)(G/K))
\]
composed with the multiplication map.
\end{itemize}
\end{Lemma}
\begin{proof}
We identify the slice category appearing in the pointwise formula for Kan extensions (\ref{eqn:PointwiseFormula}) as follows:
\begin{align*}
	\Orb_\cf^{\times}\times_{\Orb_\cf}(\Orb_\cf)_{/(G/K)}	&\simeq\coprod_{(H\in \cf)} B(W_GH)\times_{\Orb_\cf}(\Orb_\cf)_{/(G/K)}	\\
											&\simeq\coprod_{(H\in \cf)}\colim_{B(W_GH)}\Orb_\cf(-,G/K)			\\
											&\simeq\coprod_{(H\in \cf)}((G/K)^H)_{hW_GH}						\\
											&\simeq\coprod_{(H\leq K)}B(W_KH).
\end{align*}
In the first step $\orbcatinv[\cf]$ is decomposed into a disjoint union of Weyl groups (Remark \ref{Rmk:DisjointUnionWeyl}).
The second equivalence uses that the left fibration $(\Orb_\cf)_{/(G/K)}\rightarrow \Orb_\cf$
is classified by the functor $\Orb_\cf(-,G/K)$ and that the total space of a left fibration is equivalent to the colimit of its classifying diagram in spaces \cite[3.3.4.6]{htt}.
The last equivalence follows from the decomposition formula used in the proof of Lemma \ref{Lemma:FixedPointCorep}.

For part (ii) we recall the following distributivity formula for the free commutative algebra,
which follows from its description $\mathbb PX\simeq\coprod_{n\geq 0}X^{\otimes n}/\Sigma_n$ via symmetric powers \cite[3.1.3]{ha}.
Let $\{X_i\}_{i\in I}$ be a collection of objects indexed by a finite set $I$.
Then there is canonical equivalence
\[
	\mathbb P\left(\coprod_{i\in I}X_i\right)\simeq\coprod_{\alpha\in\bn^I}\left(\bigotimes_{i\in I}X_i^{\otimes \alpha_i}/\Sigma_{\alpha_i}\right)
\]
that is induced on the summands by the product of the maps $X_i\rightarrow \bp X_i$.
We now use the conjugacy classes of subgroups as the indexing set $I$ and write $M_\alpha$ for the finite $K$-set
\[
	M_\alpha=\coprod_{(H\leq K)}(K/H)^{\sqcup\alpha_H}
\]
associated with an $I$-tuple $\alpha$ and note that (\ref{Eqn:WeylGroupAutomorphism}) yields an identification
\[
	\Aut_K(M_\alpha)\cong \prod_{(H\leq K)}\Sigma_{\alpha_H}\wr W_KH
\]
of the automorphisms group of $M_{\alpha}$.
Applying the distributivity formula to part (i), we obtain the following chain of equivalences:
\begin{align*}
	\mathbb P\left(\coprod_{(H\leq K)}X(G/H)/W_KH\right)
	&\simeq\coprod_{\alpha\in\bn^I}\left(\bigotimes_{(H\leq K)}X(G/H)^{\otimes \alpha_H}/(\Sigma_{\alpha_H}\wr W_KH)\right)	\\
	&\simeq\coprod_{\alpha\in\bn^I}\left(\bigotimes_{Kj\in K\backslash M_\alpha}X(G/\Stab_Kj)\right)/\Aut_K(M_\alpha)			\\
	&\simeq\coprod_{[M]\in \Fin_K}\left(\bigotimes_{Kj\in K\backslash M}X(G/\Stab_Kj)\right)/\Aut_K(M)			\\
	&\simeq\coprod_{\substack{(\phi:K\rightarrow\Sigma_m)\\ m\geq 0}}\left(\bigotimes_{Kj\in K\backslash \{1,\ldots,m\}}X(G/\Stab_Kj)\right)/C(\phi).
\end{align*}
The last three steps amount to reindexing and regrouping terms, using that the map $\alpha\mapsto M_\alpha$ is a bijection from $\bn^I$ to isomorphism classes of finite $K$-sets, which in turn biject 
with conjugacy classes of homomorphisms $K\rightarrow\Sigma_m$ for all $m\geq 0$.
\end{proof}
\begin{Remark}
We have formulated the previous lemma in its natural generality, using the intrinsic notion of symmetric monoidal $\infty$-categories.
Ultimately, we want to use it to identify the homotopy type of free commutative orthogonal ring spectra.
Since for cofibrant orthogonal spectra the map
\[
	{E\Sigma_n}_+\smashp_{\Sigma_n} X^{\smashp n}\overset{\simeq}{\longrightarrow} X^{\smashp n}/\Sigma_n
\]
is an equivalence, the derived symmetric algebra functor agrees with symmetric algebra in
the symmetric monoidal structure on the underlying $\infty$-category.
\end{Remark}
%

We will also need the following well-known decomposition formula for the fixed points of a quotient, which is often used in equivariant homotopy theory.
\begin{Lemma}[e.g.\ see {\cite[B.17]{global}}]\label{Lemma:FixedPointsQuotient}
Let $X$ be an $(H\times K^{\op})$-space such that the right $K$-action is free. Then the projection $X\rightarrow X/K$ induces a homeomorphism
\[
	(X/K)^H\cong\coprod_{(\alpha:H\rightarrow K)}(\alpha^\ast X)^H/C(\alpha)
\]
between the $H$-fixed points of the quotient $X/K$ and a disjoint union indexed by the conjugacy classes of group homomorphisms.
\end{Lemma}

\begin{Lemma}\label{Lemma:NormRestriction}
Let $\alpha:K\rightarrow \Sigma_m$ be a transitive group homomorphism with stabilizer $H=\Stab_K 1$. There is a natural untwisting isomorphism 
\[
	N_H^K\res^K_HX\cong \alpha^\ast X^{\wedge m}
\]
for orthogonal $K$-spectra $X$.
\end{Lemma}
\begin{proof}
Let $k_1,\ldots, k_m$ be the coset representatives chosen in the definition of the norm with associated homomorphism $\Psi:K\rightarrow \Sigma_m\wr H$.
As a homomorphism to $\Sigma_m\wr K$, this is conjugate via $(\Id;k_1,\ldots, k_m)\in\Sigma_m\wr K$ to the homomorphism $(\psi,\Delta_K)$,
where $\psi=\pr_{\Sigma_m}\circ \Psi$ is the $\Sigma_m$-component of $\Psi$ and $\Delta_K$ the diagonal inclusion of $K$.
This is in turn conjugate to $\alpha$ via the permutation $i\mapsto \alpha(k_i)(1)$ and hence left translation with these elements yields the desired isomorphism
\(
	\Psi^\ast X^{\smashp m}\cong \alpha^\ast X^{\smashp m}.
\)
\end{proof}

\begin{Proposition}\label{Prop:FreeCommutative}
The natural transformation
\[
	(\mathbb P\circ \iota_!)\circ\Phi\overset{\simeq}{\lra} \Phi_{\Com}\circ\mathbb P
\]
adjoint to
\[
	\Phi \xrightarrow{\Phi\eta} \Phi\circ\mathbb U\circ\mathbb P\simeq\mathbb U\circ\Phi_{\Com}\circ\mathbb P
\]
is an equivalence.
\end{Proposition}
\begin{proof}
Let $H\leq G$ be a finite subgroup.
The following pointset level computation only depends on the $H$-homotopy type of $X$ and
hence we may assume that it is a cofibrant $H$-spectrum in the HHR-model structure.
We then obtain:


\begin{align*}
		\Phi^H({E_G\Sigma_m}_+\wedge_{\Sigma_m}X^{\wedge m})&\cong \bigvee_{(\alpha:H\rightarrow\Sigma_m)}(E_G\Sigma_m)^{\Gamma(\alpha)}_+\wedge_{C(\alpha)}\Phi^H(\alpha^\ast X^{\wedge m})	\\
		&\cong \bigvee_{(\alpha:H\rightarrow\Sigma_m)}{EC(\alpha)}_+\wedge_{C(\alpha)}\Phi^H\left(\bigwedge_{[j]\in H\backslash\{1,\ldots, m\}}N_{\Stab j}^HX\right)	\\
		&\simeq \bigvee_{(\alpha:H\rightarrow\Sigma_m)}{EC(\alpha)}_+\wedge_{C(\alpha)}\left(\bigwedge_{[j]\in H\backslash\{1,\ldots, m\}}\Phi^{\Stab j}X\right).
\end{align*}
Since $\Phi^H$ is defined pointwise via fixed points of spaces, the first isomorphism follows from Lemma \ref{Lemma:FixedPointsQuotient}.
In the second step, using Lemma \ref{Lemma:NormRestriction}, $\alpha^\ast X^{\wedge m}$ is decomposed into a product of norms 
whose geometric fixed points are then computed using their multiplicative properties (Proposition \ref{Prop:GfpMonoidal} and  \ref{Prop:GfpNorm}).
The result now follows from Lemma \ref{Lemma:FreeCommutative}.
\end{proof}

\begin{Theorem}
Let $G$ be a compact Lie group, $\mathcal F$ a family of finite subgroups of $G$ and $R\subseteq \bz$ a ring
such that the group orders of $\mathcal F$ are invertible in $R$. Then the geometric fixed point functor
\[
	\Phi^{\Com}:\Com(\gensp)_{\mathcal F, R}\overset{\simeq}{\lra} \Orb_{\mathcal F}\hyph \Com(\Sp)_R
\]
 is an equivalence.
\end{Theorem}
\begin{proof}
We consider the diagram
\begin{center}
	\begin{tikzcd}[row sep = large, column sep = large]
	\gensp_{\mathcal F, R}	\arrow[r, "\Phi"]\arrow[d, shift right = 2, "\mathbb P_{\gen}", swap]	&
	\Orb^{\times}_{\mathcal F}\hyph \Sp_R	\arrow[d, shift right = 2, swap, "\mathbb P\circ \iota_!"]		\\
	
	\Com(\gensp)_{\mathcal F, R}	\arrow[r]\arrow[u, shift right = 2, "\mathbb U", swap]	&
	\Orb_{\mathcal F}\hyph \Com(\Sp)_R	\arrow[u, shift right = 2, "\mathbb U", swap]
	\end{tikzcd}
\end{center}
of $\infty$-categories in which the square with the forgetful functors commutes.
The upper horizontal arrow is an equivalence (Theorem \ref{Theorem:AdditiveEquivalence}) and both adjunctions are monadic (Proposition \ref{Prop:MonadicAdjunction}). 
By \cite[4.7.3.16]{ha}, it suffices in this situation to show that the natural transformation
\[
	(\mathbb P\circ \Lan)\Phi X\ra \Phi_{\Com}(\mathbb P_{\gen}X)
\]
is an equivalence and this is the content of Proposition \ref{Prop:FreeCommutative}.
\end{proof}
\begin{Remark}
Similarly, a monadicity argument can also be used to show that the canonical map
\[
 \Com(\Sp_R)=N(\Com(\orthspec))[W_R^{-1}]\overset{\simeq}{\longrightarrow}\operatorname{CAlg}(\Sp_R)
\]
is an equivalence between the underlying $\infty$-category of commutative ring spectra and commutative algebras in the symmetric monoidal $\infty$-category of spectra.
This is of course well-known, but the author does not know a reference in the literature that directly applies to orthogonal spectra.
\end{Remark}


\appendix
\section{Homotopy colimits of simplicial objects}
In this appendix we briefly explain the essential steps in showing that the forgetful functor
\[
	\bu:\Com(\gensp)_{\cf,R}\rightarrow\gensp_{\cf,R}
\]
preserves colimits of simplicial objects. We first reduce to the case that $G$ is finite by noting that the restriction functors
$\res^G_H$ to finite subgroups are left adjoints that jointly detect equivalences.
We may also assume $R=\bz$, since the inclusion of $R$-local spectra preserves colimits.

For technical reasons, we need to move to a simplicial setting on the pointset level.
Let $G\hyph\Sp^\Sigma_{\operatorname{sSet}}$ ($G\hyph\Sp^\Sigma_{\operatorname{Top}}$) denote the category of $G$-symmetric spectra
in simplicial sets (topological spaces) endowed with the positive $G$-flat model structure \cite[Thm.\ 4.7]{hausmann}
(with respect to a complete representation universe), which also transfers to a model structure on commutative ring spectra.
The forgetful functor $G\hyph\orthspec\rightarrow G\hyph\Sp^\Sigma_{\operatorname{Top}}$ from orthogonal spectra
is the right adjoint of a Quillen equivalence \cite[Thm.\ 7.5]{hausmann}.
\begin{Proposition}
Let $G$ be a finite group. There is a zig-zag of Quillen equivalences
\[
	\Com(G\hyph\orthspec)\simeq_Q\Com(G\hyph\Sp^{\Sigma}_{\operatorname{sSet}})
\]
between commutative orthogonal $G$-ring spectra and symmetric $G$-ring spectra of simplicial sets.
\end{Proposition}
\begin{proof}
The forgetful functor is lax symmetric monoidal and the induced oplax monoidal structure on the left adjoint is actually an isomorphism.
In particular, this shows that the adjunction restricts to one between categories of commutative ring spectra,
which is automatically a Quillen pair because the model structures are transferred.
The right adjoint reflects weak equivalences and so it remains to show that the counit is a weak equivalence on cofibrant ring spectra. 
In the flat model structure the underlying spectra are (non-positively) cofibrant and in that case we already know that the unit is an equivalence.
We thus get a Quillen equivalence with symmetric spectra in topological spaces.
The adjoint pair geometric realization and singular complex prolongs to a further equivalence with
spectra in simplicial sets.
\end{proof}
The theory of colimits in the underlying $\infty$-category of a simplicial model category reduces to that of homotopy colimits \cite[4.2.4]{htt}
and so it will suffice to show that the forgetful functor preserves these. 
The key technical advantage of symmetric spectra in simplicial sets is that the geometric realization functor
\[
	\vert-\vert:\Fun(\Delta^{\op}, G\hyph\Sp^\Sigma)\rightarrow G\hyph\Sp^\Sigma
\]
is homotopical.
\begin{Proposition}
Let $f:X\rightarrow Y$ be an objectwise equivalence of simplicial symmetric $G$-spectra.
Then its geometric realization $\vert X\vert\rightarrow \vert Y\vert$ is also an equivalence.
\end{Proposition}
\begin{proof}
Non-equivariantly this is \cite[1.11]{harper}. The proof proceeds by inducting along the skeletal filtrations and also works equivariantly,
because homotopically it only depends on the cone of a monomorphism being equivalent to the actual quotient. In fact, they are level-homotopy equivalent.
\end{proof}
In particular, this shows that the realization 
\[
	\vert X \vert \simeq \hocolim_{\Delta^{\op}}X
\]
of a simplicial object $X$ in $G\hyph\Sp^\Sigma$ always has the correct homotopy type,
since we can Reedy-cofibrantly replace without changing it (c.f.\ \cite{hirschhorn}).

\begin{Proposition}
Let $R$ be a simplicial object in commutative symmetric $G$-ring spectra.
Then its realization can be computed in underlying category of symmetric $G$-spectra, that is, there is a natural isomorphism
\[
	U\vert R\vert\cong \vert UR\vert.
\]
\end{Proposition}
\begin{proof}
This is mostly formal and a detailed proof of an analogous statement in the topological setting is given in \cite[2.1.7]{global}, also see \cite{ekmm}.
We only recall the outline of the argument here. Using that the smash product commutes with colimits and simplicial tensors in each variable, one identifies the smash product 
\[
	\vert X\vert\smashp\vert Y\vert\cong \vert X\bar\smashp Y\vert
\]
of the realizations of two simplicial spectra $X$ and $Y$ 
as the realization of the bisimplicial object $(n,m)\mapsto X_n\bar\smashp Y_m$. This is in turn isomorphic to the realization of the diagonal
$X\smashp Y =\Delta^\ast(X\bar\smashp Y)$, i.e.\ the level-wise smash product. Combined, this yields  $\Sigma_n$-equivariant isomorphisms
\(
	\vert X\vert^{\smashp n}\cong \vert X^{\smashp n}\vert
\)
and hence an identification
\[
	\vert\bp X\vert\cong\bp\vert X\vert \cong \vert U\bp X\vert 
\]
on free commutative algebras, where the realization on the left is performed in the category of commutative ring spectra.
The general case then follows by considering the canonical coequalizer presentation
\[
	R\leftarrow\bp R\leftleftarrows\bp(\bp R)
\]
of a commutative ring spectrum $R$.
\end{proof}

\begin{Corollary}
The forgetful functor preserves homotopy colimits of simplicial objects.
\end{Corollary}
\begin{proof}
The previous discussion yields equivalences
\[
	U(\hocolim_{\Delta^{\op}}R)\simeq U\vert R\vert\cong \vert UR\vert\simeq\hocolim_{\Delta^{\op}}UR
\]
for a Reedy cofibrant simplicial object in commutative ring spectra.
\end{proof}
\bibliography{references.bib}
\bibliographystyle{alpha}
%
%

\end{document}